\documentclass[12pt]{amsart}
\usepackage[headings]{fullpage}

\usepackage{amsmath,amssymb,amsthm}
\usepackage{color}
\usepackage{graphicx}
\usepackage{enumerate}
\usepackage{eufrak}
\usepackage{url}
\usepackage{slashed}
\usepackage{bm}
\usepackage{tikz-cd}
\usepackage{pb-diagram}
\usepackage[english]{babel}

\usepackage{enumitem}
\usepackage[bookmarks=true,%
    colorlinks=true,%
    linkcolor=blue,%
    citecolor=blue,%
    filecolor=blue,%
    menucolor=blue,%
    urlcolor=blue,%
    breaklinks=true]{hyperref}

\usepackage[all]{xy}
\usepackage{verbatim}

\newcommand{\HP}{\mathbb{H}}

\newcommand{\R}{\mathbb{R}}
\newcommand{\Z}{\mathbb{Z}}

\newcommand{\Card}{\mathrm{Card\,}}
\newcommand{\Cay}{\mathrm{Cay}}

\newcommand{\codim}{\mathrm{codim}}

\newcommand{\Max}{\mathrm{Max}}

\newcommand{\PSL}{{\mathrm {PSL}}}

\newcommand{\Sys}{{\mathrm {Sys}}\,}

\newcommand{\Tess}{{\mathrm {Tess}}}
\newcommand{\Tr}{{\mathrm {Tr}}}

\newcommand{\arcosh}{\mathrm{arcosh}}
\newcommand{\Curv}{\mathrm{Curv}}
\newcommand{\Syst}{\mathrm{Syst}}
\renewcommand{\d}{\mathrm{d}}

\newtheorem{thm}{Theorem}
\newtheorem{cor}[thm]{Corollary}

\newtheorem{lemma}[thm]{Lemma}

\newtheorem*{thm*}{Theorem}
\newtheorem*{citethm}{Theorem}

\newtheorem*{Main1}{Theorem 1}

\newtheorem*{crit*}{Criterion}

\newcommand{\ch}{\operatorname{ch}}

\newcommand{\param}{{\mathchoice{\mkern1mu\mbox{\raise2.2pt\hbox{$
\centerdot$}}
\mkern1mu}{\mkern1mu\mbox{\raise2.2pt\hbox{$\centerdot$}}\mkern1mu}{
\mkern1.5mu\centerdot\mkern1.5mu}{\mkern1.5mu\centerdot\mkern1.5mu}}}

\renewcommand{\setminus}{{\smallsetminus}}

\begin{document}
\title{Estimating the dimension of the Thurston Spine}
\author{Olivier Mathieu}

\address{SUSTech International Center for Mathematics\\
Southern University of Science and Technology\\Shenzhen, China
}
\address{Institut  Camille Jordan du CNRS\\
UCBL, 69622 Villeurbanne Cedex, France
}

\email{mathieu@math.univ-lyon1.fr}

\today

\begin{abstract}
For $g\geq 2$, the Thurston spine $\mathcal{P}_{g}$ is the subspace of Teichm\"uller space $\mathcal{T}_{g}$, consisting of the marked surfaces for which the set of shortest curves, the systoles, cuts the surface into polygons. Our main result is the existence of an
infinite set $A$ of
integers $g\geq 2$ such that

\centerline{$\codim\, \mathcal{P}_g\in o(g/\sqrt{\log g})$,}

\noindent when $g\in A$ goes to $\infty$. This proves  the recent conjecture of M. Fortier Bourque.
\end{abstract}

\maketitle

{\footnotesize
\tableofcontents
}

\section*{Introduction}
\label{sect:intro}


\noindent{\it 0.1 General Introduction}

\noindent  Let $g\geq 2$.
The Teichm\"uller space $\mathcal{T}_g$ is the space of all marked closed hyperbolic surfaces of genus $g$.
(Precise definitions used in the introduction can be found in Section \ref{sub:defns}.)
It is a smooth variety homeomorphic to $\R^{6g-6}$, see e.g. \cite{Bu}\cite{FM}, on which the mapping class group $\Gamma_g$ acts properly. 
 By Harer's Theorem
\cite{H}, $\Gamma_g$ has virtual cohomological dimension
$4g-5$. This leads to the question, raised in
\cite{BV}-can we find an equivariant deformation retraction
of $\mathcal{T}_g$ onto a subcomplex  of dimension
$4g-5$, or equivalently, of codimension $2g-1$?

In a remarkable note \cite{T85}, Thurston considered the subspace
${\mathcal P}_g\subset \mathcal{T}_g$ consisting of
marked surfaces for which the systoles
{\it fill} the surface, i.e. the systoles cut the surface into polygons.  In {\it loc. cit.}, he proved
\footnote{In \cite{LiJi},
some doubts have been raised about Thurston's proof. The results stated below were clearly motivated by his note \cite{T85}, but their proofs are independent of {\it loc. cit.}.  So,  we will not discuss here if the main result of \cite{T85} is proved or not.} that
$\mathcal{P}_g$ is an equivariant deformation retract of $\mathcal{T}_g$. Since, $\mathcal{P}_g$ is called the Thurston spine. It follows from
\cite{T85} that $\dim \mathcal{P}_g\geq 4g-5$, or,
equivalently, that
$\codim \mathcal{P}_g\leq 2g-1$.
 
It was shown in Theorem 44 of \cite{SS} and verified by a Sage computation in
\cite{Calculation} that $\dim \mathcal{P}_{2}=3$, which is the virtual cohomological dimension of $\Gamma_2$. Therefore, one could have expected that
$\mathcal{P}_g$ has  codimension
$2g-1$ for all $g$. In that direction, P. Schmutz
Schaller provided examples of surfaces of genus $g$
which are cut by a minimal set of $2g$ systoles. 
(We could expect that  ${\mathcal P}_g$ has locally codimension $2g-1$, as it will be explained in Subsection 0.2.) 
A year ago, the breakthrough  paper \cite{FB23} showed that $\codim \mathcal{P}_g<2g-1$,  for infinitely many $g$. Nevertheless,  I. Irmer proved that $\mathcal{P}_g$ admits a $\Gamma_g$-equivariant deformation retract into a subcomplex of minimal dimension $4g-5$, see \cite{Me}. \\

More precisely, M. Fortier Bourque proved  in \cite{FB23}  that 
 
 \centerline{$\liminf_{g\to\infty} {\codim\, P_g/ g}\leq 1$.}
 
\noindent
Moreover  he had conjectured earlier \cite{FB20}  that

\centerline{$\liminf_{g\to\infty}\codim \mathcal{P}_g/g=0$.} 

\noindent Our  paper provides a proof of his conjecture with, in addition, some explicit bound.

\begin{thm} 
\label{boundthm}
There exists an infinite set $A$ of integers $g\geq 2$ such that

\centerline{$\codim\, \mathcal{P}_g <
{38\over \sqrt{\ln\ln\ln g}}\,\,
{g\over \sqrt{ \ln g}}$,}

\noindent for any $g\in A$.
\end{thm}

This leads to the concrete question -which is the smallest $g$ for which $\codim\, \mathcal{P}_g <2g-1$?
In the last section, we will see that, for $g=17$, we have $\codim\, \mathcal{P}_{17} <32$. However, we do not know if $g=17$ is the smallest $g$ 
answering the question.

\bigskip
\noindent
{\it 0.2 Organization of the paper and the main idea of the proof}

\noindent  
The starting point of the proof is based on the main result of \cite{IM1}, that we now recall.

A regular right-angled hexagon $H$ of the Poincar\'e half plane $\HP$ is called 
{\it decorated} if it is oriented and its sides are cyclically indexed by $\Z/6\Z$. Up to  direct isometries, there are two such hexagons 
$\mathcal{H}$ and $\overline {\mathcal{H}}$, with opposide orientations.

A tesselation of a closed oriented hyperbolic surface $S$ is called a {\it standard  tesselation}
if each  tile is isometric to   $\mathcal{H}$ or $\overline {\mathcal{H}}$. Of course, it is presumed that the tiles are glued along edges with the same index, therefore a tile isometric to $\mathcal{H}$ is surrounded by six tiles isometric to $\overline {\mathcal{H}}$ and conversely.
  A vertex of a standard tesselation is an intersection point of two perpendicular geodesics.
Therefore, the $1$-skeleton of a standard tesselation consists of a finite family of closed geodesics, called the {\it curves} of the tesselation.

 For   a  hyperbolic surface $S$, denote by $\Syst(S)$ the set of systoles of $S$.\\

\begin{citethm}[Theorem 25 of \cite{IM1}]
There exists an infinite set $A$ of integers $g\geq 2$, and, for any $g\in A$ a closed oriented hyperbolic surface $S_g$ of genus $g$ endowed with a standard  tesselation $\tau_g$ such that

\begin{itemize}
\item[(1)] the  systoles of $S_g$ are exactly the curves of 
$\tau_g$, and
\item[(2)] we have 
$$\Card\,\Syst(S_g)\leq {57\over \sqrt{\ln\ln\ln g}}\,\,
{g\over \sqrt{ \ln g}} .$$
\end{itemize}
\end{citethm}

The {\it index} of a curve of the tesselation $\tau_g$ is the common index of its edges.
It is clear that the subset
 ${\mathcal C}$  of curves of index
$\neq 1$ or $2$ fills the surface and
$\Card\,{\mathcal C}=2/3\, \Card \Syst(S_g)$. 

Let $\Sys({\mathcal C})$ be the set of marked hyperbolic surfaces $S$ of genus $g$ such that 
$\Syst(S)={\mathcal C}$. For any curve, or free homotopy class, $C$ of $S_g$, let $L(C)$ be its length,  viewed as a function on the Teichm\"uller space $\mathcal{T}_{g}$. 
Set ${\mathcal C}=\{C_1, C_2\dots\}$.
The subspace $\Sys({\mathcal C})$ is defined
by the $\Card\,{\mathcal C}-1$ equations

\centerline{$L(C_1)=L(C_2)=\dots$}

\noindent together with some inequalities.
Intuitively, our result should follow from the
following  two facts 

\begin{itemize}
\item[(1)] In ${\mathcal T}_g$, the point $S_g$
is adjacent to $\Sys({\mathcal C})$, and

\item[(2)] for any point $x\in S({\mathcal C})$ closed to $S_g$, we have
$\codim_x\, S({\mathcal C})=\Card\,{\mathcal C}-1$
\end{itemize}

If we assume that
the differentials $\{\d L(C)\mid c\in\Syst(S_g)\}$, are linearly independent at the point $S_g$, the 
previous two facts would follow from 
the submersion theorem.
However, an argument  in Theorem 36 of \cite{SS} shows that these differentials  are often linearly dependent at $S_{g}$. 
For this reason, the cardinality of ${\mathcal C}$ does not determine  the local codimension of $\Sys({\mathcal C})$. 

Following an idea of Sanki \cite{S}, we can deform the angles of the tiles, by alternately replacing the right angles by angles of value $\epsilon$ and $\pi- \epsilon$, for any $\epsilon\in]0,\pi[$.  In this way we obtain a path 
$\sigma:]0,\pi[\to \mathcal{T}_g$, such that $\sigma(\pi/2)$ is the hyperbolic surface $S_g$. We will called it the
{\it Sanki} path of the tesselation $\tau_g$. The main idea of the proof is the following

\begin{thm} [see Section \ref{Gloop}] 
The set of differentials

\centerline
{$\{\d L(C)\ | C\in \Syst(S_g)\}$}

\noindent is linearly independent at $\sigma(\epsilon)$, except for finitely many values of $\epsilon$. 
\end{thm}

It implies that the assertions (1) and (2) are correct for $x=\sigma(\epsilon)$, where
$\epsilon\neq\pi/2$ is closed enough to $\pi/2$.

The proof of Theorem 2 is based on a  duality,
which is expressed in terms of the Poisson product associated with the Weil-Petersson symplectic structure on ${\mathcal T}_g$.  For any 
curve $B$ of the tesselation, we define a dual function $L(B^*)$ which is a linear combination (with coefficients $\pm 1/2$) of  lengths of three curves, which are the boundary components of a well-chosen pair of pants. 

It results from an  asymptotic analysis of Wolpert's formula \cite{wolpert83} that

\begin{equation}
 \lim_{\epsilon\to 0}\,\,
 \{L(B),L(A^*)\}(\sigma(\epsilon))=\delta_{A,B}
 \end{equation}
 
\noindent for any two curves
 $A$, $B$ of the tesselation, where,
as usual, $\delta_{A,B}$ denotes the Kronecker's symbol. Since $\delta:=\det (\{L(B),L(A^*)\})$ is an analytic function, it follows that 
$\delta(\sigma(\epsilon))$ is not zero for any
$\epsilon\neq\pi/2$ closed to $\pi/2$.

In fact the proof of equation (1) is based on 
elementary but lengthy trigonometric computations
of Sections 2-4.
To present the computations
and the figures  as simply as possible,
we have restricted ourself to hexagonal tesselations. However similar results hold for
tesselations by $2p$-gons for any $p\geq 3$.







\section{Background and Definitions}
\label{sub:defns}

\noindent
{\it 1.1 Marking of surfaces and the Teichm\"uller space}

\noindent 
Let $g\geq 2$. By definition, the {\it Teichm\"uller} space ${\mathcal T}_g$ is the space of
all marked oriented closed  hyperbolic surfaces of genus $g$. It means that ${\mathcal T}_g$ parametrizes the set of those hyperbolic surfaces, where the marking is a datum that distinguishes isometric surfaces corresponding to distinct paramaters. 

There are various equivalent definitions of the marking \cite{Bu}. Here we will adopt the
most convenient for our purpose. Let
$\Pi_g$ be the group given by the following presentation

$$\langle a_1,b_1,\dots,a_g,b_g
\mid (a_1,b_1)(a_2,b_2)\dots (a_g,b_g)=1\rangle .$$

Then a point $x$ of the Teichmuller space is a loxodromic (i.e. faithfull with discrete image) representation $\rho_x:\Pi_g\to\PSL(2,\R)$ modulo
linear equivalence. At the point $x$, the corresponding hyperbolic surface is 
$S_x:=\HP/\rho_x(\Pi_g)$. With this definition,
${\mathcal T}_g$ is a connected component of a
real algebraic variety.

Formally, a {\it curve} $c$ is a nontrivial conjugacy class of $\Pi_g$. For a hyperbolic surface, any
free homotopy class has a unique geodesic representative. Thus $c$  defines a closed geodesic
$c_x$ of $S_x$, for any  $x\in {\mathcal T}_g$.
Here closed geodesics are nonoriented, so we will not distinguish the conjugacy classes of $c$ and $c^{-1}$.
Concretely, the marking of the surface $S_x$ means that each geodesic $C$ of $S_x$ is marked by a conjugacy class in $\Pi_g$.

In this setting, the \textit{mapping class group} $\Gamma_{g}$ is the group of all outer automorphisms of $\Pi_g$ which act trivially on $H^2(\Pi_g)\simeq\Z$.
It  acts on $\mathcal{T}_{g}$ by 
changing the marking, or, more formally, by  twisting the representation of $\Pi_g$.

\smallskip\noindent
{\it 1.2 Length of curves}

\noindent The length  of an arc, or a closed geodesic, $e$ will be denoted $l(e)$. When there is no possibility of confusion, we will use the same letter for an arc $e$ and its length. For example the expression $\cosh\, H$ in the proof of Lemma \ref{inequality} stands for $\cosh\, l(H)$.

Let $x\in{\mathcal T}_g$.
Given a curve $c$, set $L(c)(x)=l(c_x)$  where
$c_x$ is  the geodesic representative
of $c$ at $x$.  
The formula $2\ch (L(c(x)))=\vert\Tr(\rho_x(c))\vert$ shows that the function 
$L(c):{\mathcal T}_g\to\R$ is analytic. Let $C$ be a closed geodesic of $S_x$.
We set $L(C)=L(c)$, where $c$ is the curve marking $C$.

\smallskip\noindent
{\it 1.3 The Thurston's spine $\mathcal{P}_g$} 

\noindent
In riemannian geometry,
a {\it systole} is an essential closed geodesic of
minimal length. In fact, for a hyperbolic surface, any closed geodesic is essential.
Let $\mathcal{P}_g$ be the set of all points $x\in \mathcal{T}_g$ such that the set of systoles fills $S_{x}$, i.e. it cuts $S_{x}$ into polygons. The subspace $\mathcal{P}_g$ is called the {\it Thurston spine}, see \cite{T85}.

By definition, the Thurston spine $\mathcal{P}_g$ is a semi-analytic subset \cite{T85}, and  therefore 
it admits a triangulation by \cite{L1}. In particular,
the  dimension $\dim_x\, \mathcal{P}_g$ at any point $x\in \mathcal{P}_g$ is well defined. Set

$$\dim\, \mathcal{P}_g=
\Max_{x\in \mathcal{P}_g}\,\dim_x\, \mathcal{P}_g\text{ and }\codim \mathcal{P}_g:=
\dim\,\mathcal{T}_g-\dim\,\mathcal{P}_g.$$

 \bigskip
\noindent
{\it 1.4 Orientation of the boundary components}

\noindent In what follows, all surfaces $S$ are given with an orientation. When $S$ has a boundary $\partial S$, it is
 oriented by the rule that, while moving forward along 
$\partial S$,    the interior of 
$S$ is on the  right. 
With this convention, when a circle of the plane
is viewed as the boundary of its interior, it 
is oriented in the clockwise direction.

\bigskip
\noindent
{\it 1.5 Angles}

\noindent Let $C$, $D$ be two distinct geodesic arcs of a surface and let $P$ be an intersection point. The angle of $C$ and $D$ at $P$, denoted 
$\angle_P CD$ is measured anticlockwise from $T_PC$ to $T_PD$, where $T_PC$ and $T_PD$ are the tangent line at $P$ of $C$ and $D$. By definition
$\angle_P CD$ belongs to $]0,\pi[$. When we permutes $C$ and $D$ there is the formula

\centerline{ $\angle_P DC=\pi-\angle_P CD$.}

\noindent In what follows, it will be convenient to set $\overline\alpha=\pi-\alpha$ for any $\alpha\in [0,\pi]$. Also it will be convenient to use the notation
$\angle DC$ when  the point $P$ is unambiguously defined.

The notion of {\it inner angles} is different.
Let $S$ be a surface whose boundary 
$\partial S$ is piecewise
geodesic. Let $c$ and $d$ be two consecutive 
geodesic arcs of $\partial S$ meeting at a point $P$. The inner angle is a real number 
$\alpha\in]0,2 \pi[$. We have $\alpha<\pi$ when $S$ is locally convex around $P$.
In that case, the equality 
$\angle\,cd=\alpha$ means that the arc $c$ preceeds $d$ when going forward along $\partial S$.

\bigskip
\noindent











\section{Trigonometry in $\HP$}

\noindent As stated in the Introduction, the analysis of Sanki's paths, defined in Section \ref{Gloop}, is based on many trigonometric computations. This section involves trigonometric computations in the Poincar\'e half-plane $\HP$.
Subsequent computations in the pairs of pants $\Pi(k,\epsilon)$ will be done
in Section \ref{Pants}.\\

 Let   $d_\HP$ be the hyperbolic distance
on $\HP$. By definition, a {\it line}   is  a complete geodesic $\Delta$ of $\HP$.
For any $P, Q\in\Delta$, the closed arc  between $P$  and $Q$ is called a {\it segment} and it will be  denoted $PQ$. When necessary, the segment $PQ$  is oriented from $P$ to $Q$.

Given three points $A,B$ and $C\in\HP$, we denote by $ABC$ the triangle $T$ whose sides are $AB$, $BC$ and $CA$. By our convention, $\partial T$ is oriented clockwise, but we do not require a specific orientation of the sides. Given four points $A,B,C$ and $D\in\HP$, we define in the same way the quadrilateral $ABCD$, not ruling out the possibility that one pair of opposite sides intersects.

For the whole section, we will fix an angle
$\epsilon\in]0,\pi[$. In the pictures, we will assume that $\epsilon<\pi/2$.

\bigskip
\noindent
{\it 2.1 The $\epsilon$-pencil $\mathcal{F}_\epsilon(\Delta)$ in $\HP$}

\noindent  Let $\Delta\subset\HP$ be a line. For any $P\in\Delta$, let $F(P)$ be the line passing through $P$ with $\angle \Delta F(P)=\epsilon$, see Section 1.5 for our convention concerning angles. Since no triangle has two angles of values $\epsilon$ and $\overline\epsilon$, any two lines $F(P)$ and $F(P')$ are parallel. The set
$\mathcal{F}_\epsilon(\Delta)
:=\{F(P)\mid P\in\Delta\}$ will be called the {\it $\epsilon$-pencil} along the line $\Delta$. \\

Let $\Delta'$ be a geodesic arc of $\HP$ with $\Delta\cap \Delta'=\emptyset$. When $F(P)$ meets $\Delta'$, set $\Omega(P)=F(P)\cap \Delta'$ and  

\centerline {$\omega(P):=\angle \Delta' F(P)$.}

\noindent The line $\Delta$ is oriented by the convention that, while going forward, the arc $\Delta'$ is on the right. Therefore the notion of an increasing function 
$f:\Delta\to\R$ is well-defined.\\

In contrast with the Euclidean geometry, the angle $\omega(P)$ is not constant. On the contrary, it can vary from $0$ to $\pi$, as it will now be shown.

\begin{lemma}
 \label{pencil} 
We have:
\begin{enumerate}
\item The set $I:= \{P\in \Delta\mid F(P)\cap \Delta' \neq\emptyset\}$ is an interval of $\Delta$. Moreover, the map $\omega:I\to ]0,\pi[$ is increasing.

\item Furthermore, if $\Delta'$ is a line of $\HP$ and $d_\HP(\Delta',\Delta)\neq 0$, then $\omega$ is  bijective. 
\end{enumerate}
\end{lemma}

\noindent
{\it Proof of claim (1).}
Let PP' be a positively oriented arc of $I$. Consider the quadrilateral
$Q:=P P'\Omega(P')\Omega(P)$. With our conventions, its inner angles are $\epsilon,\overline{\epsilon},
\overline{\omega(P')}$ and $\omega(P)$. Since
the area of $Q$ is

\centerline{
$2\pi -[\epsilon+\overline{\epsilon}+
\overline{\omega(P')}+\omega(P)]=
\omega(P')-\omega(P)$,}

\noindent the function $\omega$ is increasing. \\

Next let $P''$ in the interior of the segment
$PP'$. Since the line $F(P'')$ enters $Q$ at $P''$,
it should left $Q$ at another point.
Since $F(P'')$ is parallel to $F(P)$ and $F(P')$, the
exit point lies in the segment
$\Omega(P)\Omega(P')$, hence $P''$ belongs to $I$. 
Since $I$ contains a segment whenever it contains its two extremal points,  $I$ is an interval.\\

\smallskip\noindent
{\it Proof of Claim (2).}
We can assume that $\Delta=\R i$. By the assumption that $d_\HP(\Delta,\Delta')>0$, the lines  $\Delta$ and $\Delta'$ do not intersect and their endpoints in $\partial\HP$ are distinct. Therefore the endpoints $a,\,b$ of $\Delta'$ in $\partial\HP$ are real numbers with same signs. Without loss of generality, we can also assume that $0< a< b$,  as shown in Figure 1.\\

There is a line $F^+$ (resp. $F^-$) in $\mathcal{F}_\epsilon(\Delta)$ 
whose the endpoint in $\R_{>0}$  is  $b$ 
(resp.  $a$). Set

\centerline {$P^\pm=
\Delta\cap F^\pm$.}

\bigskip
\begin{figure}[ht!]\label{FIG1}
\centering
\includegraphics[width=90mm]{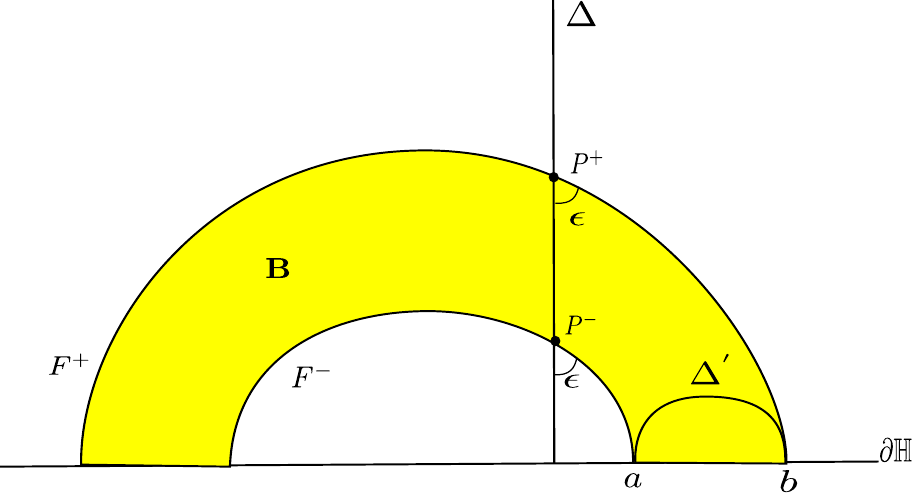}
\caption{The band ${\bf B}$}
\end{figure}

Let ${\bf B}$ be the open band delimited by $F^+$ and $F^-$. When $P$ belongs 
to the interior of $P^-P^+$, the line $F(P)$ lies in the interior of the band ${\bf B}$ and $F(t)$ meets $\Delta'$. It is clear from the definition of $F^\pm$ that

$$\lim_{P\to P_-}\,\omega(P)=0\,\,{\mathrm and}
\lim_{P\to P_+}\,\omega(P)=\pi.$$

\noindent Hence by Assertion (1),  $\omega$ is bijective.

\qed

\bigskip
\noindent
{\it 2.2 The $\epsilon$-edge}

\noindent Let $\epsilon\in]0,\pi[$ and let $\Delta$, $\Delta'$ be two lines with $d_\HP(\Delta',\Delta)>0$. Let $H$ be the  the common perpendicular arc to $\Delta$ and $\Delta'$ and let $S\in \Delta$ and $S'\in\Delta'$ be its endpoints.\\

By Lemma \ref{pencil} there is a unique $P\in\Delta$ such that $\omega(P)=\epsilon$. The edge $e =P\Omega(P)$ will be called the {\it $\epsilon$-edge} of $\Delta$ and $\Delta'$. For $\epsilon=\pi/2$, the $\epsilon$-edge is the perpendicular arc $H$.

\begin{lemma}\label{inequality} Set 
$L=d_\HP(P,P')$ where $P'=\Omega(P)$.
If $\epsilon\neq \pi/2$, then
\begin{enumerate}
\item $H$ and $e$ intersect at their midpoints.  

\item $d_\HP(P, S)=d_\HP(P', S')< L/2$.
\end{enumerate}
Moreover the segment  $SP$ is positively oriented whenever $\epsilon<\pi/2$.
\end{lemma} 

\begin{proof} 
We have $\angle \Delta e=\epsilon$ and $\angle e\Delta'=\overline \epsilon$. The sum of the four angles of the quadrilateral
$Q:=SP\Omega(P)S'$ is $2\pi$. It follows that a pair of opposite edges must intersect, so $e$ meets $H$ at some point $M$. The two triangles $SPM$ and $M\Omega(P)S'$ have the same three angles, and are therefore isometric. In particular 
$d_\HP(P, S)=d_\HP(P', S')$.\\

It follows that $e$ and the arc $H=SS'$ intersect at their midpoint $M$, thus we have 
$d_\HP(P,M)=L/2$. Since $PM$ is the hypothenuse
of the right-angled triangle $PSM$, we have
$d_\HP(P, S)< L/2$. The second claim follows.
\end{proof}

\bigskip
\noindent
{\it 2.3 Hexagon trigonometry and edge colouring}

\begin{lemma}\label{hexagon} Up to isometry, there exists a unique hyperbolic hexagon 
$H(\epsilon)$ whose sides all have the same length $L=L(\epsilon)$  and whose inner angles are alternately
$\epsilon$ and $\overline \epsilon$.\\

Moreover we have

$$\cosh L= 1+{1\over \sin\epsilon}.$$

\end{lemma}

\bigskip
\noindent
{\it Proof of the existence of $H(\epsilon)$ and of the formula for  $\cosh L$.}

\noindent Let $T$  be an oriented triangle whose angles are
$\epsilon/2$, ${\overline \epsilon}/2$ and $\pi/3$
and let $X$ be the vertex
at the $\pi/3$-angle.
Let  $\overline T$ be the triangle isometric  to $T$ with opposite orientation. Let $H(\epsilon)$ 
be the hexagon obtained by alternately gluing three copies  of $T$ and three copies of $\overline T$ around $X$. This hexagon satisfies the required conditions, so the existence is proved.\\

By the law of cosines for the triangle $T$, we have

$$\cosh L=
{\cos \epsilon/2 \cos{\overline \epsilon}/2
+\cos \pi/3 \over \sin \epsilon/2 \sin{\overline \epsilon}/2}$$

 $$=1+ {1\over 2 \sin \epsilon/2 \cos \epsilon/2}$$   

$$=1+ {1\over \sin\epsilon}.$$ 

\bigskip
\noindent {\it Proof of the uniqueness.}

\noindent Let $H$ be a hexagon whose sides all have the same length l, and whose angles are alternately $\epsilon$ and $\overline \epsilon$. Let $(A_i)_{i\in\Z/6\Z}$ be the six vertices, arranged in a cyclic order. Moreover, we will assume that the angles at $A_2$, $A_4$ and $A_6$ are $\overline\epsilon$. \\

The triangles $A_1A_2A_3$, $A_3A_4A_5$ and $A_5A_6A_1$  have the same angles at the point $A_2$, $A_4$ and $A_6$, and the two sides originating from these points have the same length $l$. Hence they are isometric. It follows that
the triangle $A_1A_3A_5$ is equilateral.\\

Let $X$ be the center of $A_1A_3A_5$, and let $T'=A_1XA_2$ and ${\overline T}'=A_2A_3X$. Since $T'$ and ${\overline T}'$ have same side lengths they are isometric, with opposite orientations. Hence $H$ is obtained by alternately gluing three copies  of $T'$ and three copies of $\overline T'$ around $X$. It follows that the angles of $T'$  are $\epsilon/2$, ${\overline \epsilon}/2$ and $\pi/3$. Therefore $T'$ is isometric the the triangle $T$ of the existence proof. Hence $H$ is isometric to $H(\epsilon)$, proving uniqueness. \qed
\bigskip

We can alternately assign the colours blue and red to the sides of $H(\epsilon)$, as follows:
If $S$, $S'$ are consecutive sides with $\angle \mathrm{S\,S'} =\epsilon$, then $S$ is red and $S'$ is blue, or, quickly speaking,
$\angle \mathrm{red\,blue} =\epsilon$. For $\epsilon\neq \pi/2$ this colouring is unique.

\bigskip
\noindent
{\it 2.4 The Saccheri quadrilateral in $H(\pi/2)$}

\noindent In the hexagon $H(\pi/2)$, let $S_1, S_2$ and $S_3$ be three consecutive sides and let $D$ be the arc joining the  vertex of $S_1\setminus S_2$ and the vertex of $S_3\setminus S_2$. Thus $S_1$, $S_2$, $S_3$ and $D$ are the four sides of a Saccheri quadrilateral. Set $L'=l(D)$.

\begin{lemma}\label{saccheri}   We have 

\centerline{$L'<2L(\pi/2)$.}
\end{lemma}

\begin{proof}  Set $L=L(\pi/2)$. The perpendicular line at the middle of $S_2$ cuts
the Saccheri quadrilateral  into two isometric Lambert quadrilaterals. It follows that

\centerline{$\sinh L'/2=\cosh L/2\,\sinh L/2 $,}

\noindent or equivalently
$\sinh^2 L'/2=4 \sinh^2 L/2$, which implies that  $\cosh L'=5$. On another hand  $\cosh 2L=2 \cosh^2 2L -1=7$. It follows that  $L'<2L$.
 \end{proof}


\section{The three-holed sphere $\Sigma(k,\epsilon)$ and the pair of  pants $\Pi(k,\epsilon)$}

\noindent
From now on, let $k\geq 3$ be an integer and
let $\epsilon\in ]0,\pi[$. We will often think of $\epsilon$ as being an acute angle, as in the figures of this section.

In this section, we will define a pair of pants
$\Pi(k,\epsilon)$ endowed with a certain tesselation.
We will  first consider a three-holed sphere 
$\Sigma(k,\epsilon)$ which has
two geodesic boundary components $C$ and $C'$ and one 
piecewise   geodesic boundary component $\mathcal{D}$.
Since the inner angles of  $\mathcal{D}$ are $<\pi$,
it is homotopic to a unique geodesic $D$.
Then  $\Pi(k,\epsilon)$ is the pair of pants lying in
 $\Sigma(k,\epsilon)$, whose boundary components are
 $C, C'$ and $D$.

\bigskip
\noindent
{\it 3.1 The tesselated three-holed sphere $\Sigma(k,\epsilon)$}

\noindent
Let us start with the planar graph $\Gamma$ represented in Figure 2.
\begin{figure}[ht!]
\centering
\includegraphics[width=90mm]{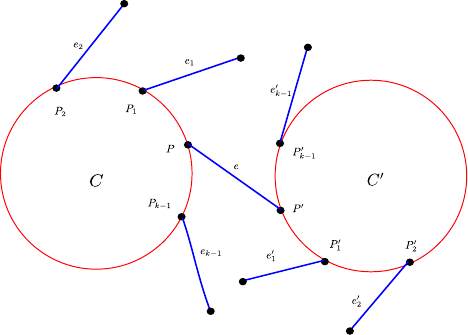}
\caption{The graph $\Gamma$}
\end{figure}

\noindent It consists of two cycles $C$ and $C'$ of length $k$, which are connected by an edge $e$ with endpoints $P\in C$ and $P'\in C'$. Starting from $P$ in an anticlockwise direction, the other points of $C$ are denoted by $P_1,\dots, P_{k-1}$. For each $1\leq i\leq k-1$ there is an additional edge $e_i$, pointing outwards from $C$. One  endpoint of $e_i$ is  $P_i$ and the other endpoint has valency one. The vertices $P'_1,\dots, P'_{k-1}$ of $C'$ and the edges $e'_i$ are defined similarly.\\

The edges of the cycles $C$ and $C'$ are coloured in red, the other edges are coloured in blue. Now we will attach $2k-2$-hexagons $H(\epsilon)$ along the edges of $\Gamma$ to obtain a hyperbolic surface $\Sigma(\epsilon,k)$. It will be convenient to define a metric on $\Gamma$ by requiring that each edge, red or blue, has length $L=1+1/\sin\epsilon$. \\

First we attach two copies of $H(\epsilon)$ along five consecutive edges of $\Gamma$. The first copy,  denoted 
$H_1(\epsilon)$, is glued along the edges $e_{k-1},\,P_{k-1}P,\,e,\,P'P'_1$ and $e'_1$. The
second copy, denoted $H_2(\epsilon)$, is glued along the edges $e'_{k-1},\,P'_{k-1}P',\,e,\,PP_1$ and $e_1$, see Figure \ref{step1}.

\begin{figure}[ht!]
\centering
\includegraphics[width=90mm]{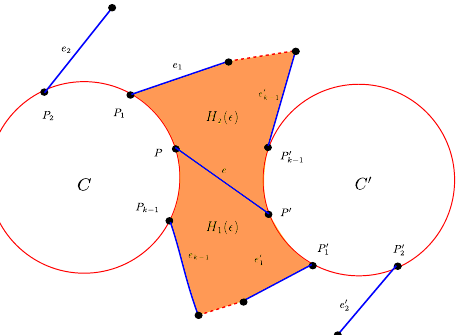}
\caption{Gluing two the first two copies $H_1(\epsilon)$ and $H_2(\epsilon)$ to 
$\Gamma$}
\label{step1}
\end{figure}

Next we attach the remaining $2k-4$  hexagons along three edges.
Indeed, for each integer $i$ with $1\leq i\leq k-2$, we glue one copy of $H(\epsilon)$ along the edges
$e_i,\,P_iP_{i+1}$ and $e_{i+1}$ and, symmetrically, we glue another copy along $e_i',\,P_i''P_{i+1}$ and $e_{i+1}'$, see Figure \ref{step2}.\\

\begin{figure}[ht!]
\centering
\includegraphics[width=90mm]{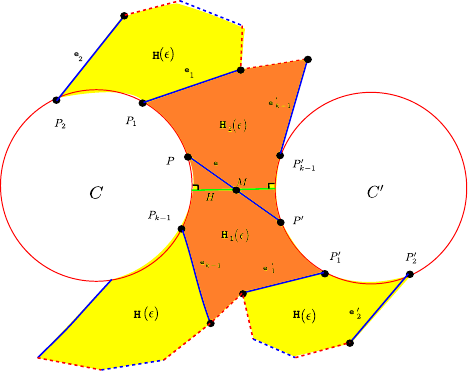}
\caption{Gluing the remaining hexagons 
$H(\epsilon)$ to 
$\Gamma$}
\label{step2}
\end{figure}

It is tacitly assumed that all gluings respect the metric and the colours of edges.\\

This defines a hyperbolic surface $\Sigma(k,\epsilon)$ which is homeomorphic to the 3-holed sphere $S_{0,3}$. Two boundary boundary components $C$ and $C'$ of $\Sigma(k,\epsilon)$ are geodesics. The third component, call it $\mathcal{D}$, is piecewise geodesic.

\begin{lemma}\label{geodesicD} The boundary component $\mathcal{D}$ is freely homotopic to a unique geodesic $D$.
Moreover
\begin{enumerate}
\item $D$ lies in the interior of $\Sigma(k,\epsilon)$,

\item $D$ meets each arc $e_i$ and $e'_i$ exactly once, and

\item $D$ does not intersect $e$.

\end{enumerate}
\end{lemma}

\begin{proof} The curve $\mathcal{D}$ is piecewise geodesic. It is alternatively composed of $2(k-2)$ blue arcs and $2(k-2)$ red arcs. Since $k\geq 3$, $\mathcal{D}$ is not a geodesic. The inner angles of $\mathcal{D}$ are each less than $\pi$, therefore $\mathcal{D}$ is freely homotopic to a unique geodesic $D$ lying in the interior of $\Sigma(k,\epsilon)$.

Let $d$ be a blue edge. There is a 1-parameter family of curves $(\mathcal{D}_t)_{t\in[0,1]}$, realizing a homotopy from $\mathcal{D}=\mathcal{D}_0$ to $D=\mathcal{D}_1$, such that the number of bigons formed by $d$ and $D=\mathcal{D}_t$ does not increase. Since there are no bigons formed by $d$ and $\mathcal{D}_0$, the geometric intersection numbers $i(d,\mathcal{D}_t)$ is constant, which proves the last two claims.
\end{proof}

\bigskip
\noindent 
{\it 3.2 The pair of pants $\Pi(k,\epsilon)$ and its
central octogon ${\bf Q}(\epsilon)$}

\noindent The geodesic $D$ decomposes $\Sigma(k,\epsilon)$ into two pieces. The component
$\Pi(k,\epsilon)\subset \Sigma(k,\epsilon)$  with geodesic  curves $C$, $C'$ and $D$ is a pair of pants.

The blue arcs decompose $\Pi(k,\epsilon)$ into two hexagons adjacent to the central edge $e$ 
and $2k-4$ quadrilaterals. Let  
${\bf H}_1(\epsilon):=H_1(\epsilon)\cap 
\Pi(k,\epsilon)$
and 
${\bf H}_2(\epsilon):=H_2(\epsilon)\cap 
\Pi(k,\epsilon)$ be the two hexagons of the decomposition. Their union 
${\bf Q}(\epsilon):={\bf H}_1(\epsilon)
\cup {\bf H}_2(\epsilon)$ is a convex octogon.\\

Let $H$ be the unique perpendicular arc joining $C$ and $C'$ and let $S\in C$ and $S'\in C'$ be its endpoints. For $\epsilon=\pi/2$, we have $H=e$. Otherwise $H$ and $e$ meet as shown in the next lemma.

\begin{lemma}\label{meeting} The arc $H$ lies in the 
octogon ${\bf Q}(\epsilon)$. When $\epsilon\neq\pi/2$,
$H$ and $e$ intersect at their midpoint.

Moreover we have   
\begin{enumerate}
\item $d(P,S)<L/2$, and

\item for $\epsilon<\pi/2$, the point $P$ belongs to ${\bf H}(\epsilon)$.
\end{enumerate}
\end{lemma}

\begin{proof} 
The inner angles of ${\bf Q}(\epsilon)$ are less than $\pi$, hence there is an isometric embedding  

\centerline{$\pi:{\bf Q}(\epsilon)\to \HP$.}

\noindent Let $\Delta$ and $\Delta'$ be the lines of $\HP$ containing, respectively, the arcs $\pi\bigl(C\cap {\bf Q}(\epsilon)\bigr)$ and $\pi\bigl (C'\cap {\bf Q}(\epsilon)\bigr)$. Let $\overline H$ be the common perpendicular arc to $\Delta$ and $\Delta'$ and let $\overline S:= \overline H\cap \Delta$ and $\overline S':= \overline H\cap \Delta'$ be its feet. By Lemma \ref{inequality}, we have

\centerline{$d_\HP(\pi(P),\overline S)<L/2$.}

\noindent 
Since $\pi\bigl(C\cap {\bf Q}(\epsilon)\bigr)$ is the geodesic arc of $\Delta$, centered at $\pi(P)$, of length $2L$, it follows that $\overline S$ is on the boundary of $\pi\bigl(C\cap {\bf Q}(\epsilon)\bigr)$. Similarly $\overline S'$ is on the boundary of $\pi\bigl(C\cap {\bf Q}(\epsilon)\bigr)$. By convexity, the arc $\overline H$ lies in $\pi\bigl({\bf Q}(\epsilon)\bigr)$. \\

Hence $H=\pi^{-1}(\overline H)$ belongs to ${\bf Q}(\epsilon)$. The other claims follow from Lemma \ref{inequality} and the fact that $\pi$ is an isometry.
\end{proof}

\section{Trigonometry in the pair of pants $\Pi(k,\epsilon)$}\label{Pants}

\noindent
Let $k\geq 3$ be an integer and let $\epsilon\in ]0,\pi[$.  

The pair of pants $\Pi(k,\epsilon)$ from the previous section is endowed with an $\epsilon$-edge $e$ joining $C$ and $C'$, whose endpoints are $P\in C$ and $P'\in C'$. Recall that the length of $e$ is
$L=\arcosh(1+1/\sin\,\epsilon)$.
Let $H$ (resp. $h$, resp. $h'$) be the unique common perpendicular arc to $C$ and $C'$ (resp. to $C$ and $D$, resp. to $C'$ and $D$). Cutting
$\Pi(k,\epsilon)$ along $H\cup h\cup h'$ provides 
the usual decomposition of the pair of pants into two right-angled hexagons.

 \bigskip\noindent
 {\it 4.1 Formula for $\cosh H$} 

\noindent
As usual, we will use the same letter for an arc and for its length. As a matter of notation, let $S\in C$ and $S'\in C'$ be the endpoints of
$H$.

\begin{lemma}\label{chH} We have

$$\cosh H=1+\sin\epsilon.$$

 \end{lemma}

\begin{proof} When $\epsilon=\pi/2$, we have $e=H$
and $\cosh  H=\cosh L=2$. Therefore we can assume that $\epsilon\neq \pi/2$.\\

By Lemma \ref{meeting}, $e$ and $H$ belong to the octogon
${\bf Q}(\epsilon)$ and intersect at their midpoint $M$. It follows that
 $SM=H/2$ and $PM=L/2$.
By the sine law applied to the triangle $PSM$, 
we have

$$\sinh H/2=\sin \epsilon \sinh L/2.$$  

\noindent
From the identities $2\sinh^2 H/2=\cosh H -1$ and
$2\sinh^2 L/2=\cosh L -1$, it follows that

$$\cosh H=1+ \sin^2\epsilon\, (\cosh L -1).$$

\noindent By Lemma \ref{hexagon}, we have 
$\cosh L-1= 1/\sin\epsilon$, from which it  follows
that $\cosh H=1+\sin\epsilon$. 
\end{proof}

\bigskip
\noindent
{\it 4.2 Conventions concerning the asymptotics of angle functions}

\noindent In what follows, we will consider
analytic functions $f:]0,\pi[\to\R$. In order to
study their asymptotic growth near  $0$,
we will use the following simplified notations. For any pair of functions
$A,\,B: ]0,\pi[\to\R$,
the expression

$$A\sim B$$

\noindent means that $$\lim_{\epsilon\to 0}{A\over B}=1.$$ Moreover the expression

$$A\sim *B$$

\noindent means that $A\sim a B$ for some positive real number $a$. Similarly, the expression
$A<<B$ means that

$$\lim_{\epsilon\to 0}{A\over B}=0. $$

\bigskip\noindent
{\it 4.3 Length estimates}

  \begin{lemma}\label{length}  We have 
 
\centerline{ $\cosh L\sim *\epsilon^{-1}$,} 
 
\centerline{ $H\sim *\epsilon^{1/2}$,}
 
 \centerline{$d(S,P)=L/2+o(1)$. }
 
 \noindent Moreover, the lengths  $h$ and $h'$
 are equal and we have
 
 \centerline{$ h\sim *\epsilon^{k-1\over 2}$.}
 
 \end{lemma}
 
 \begin{proof} Since by definition
  $\cosh L=1+1/\sin \epsilon$,  we have
 $\cosh L\sim *\epsilon^{-1}$. By Lemma 
 \ref{chH}, we have $H\sim *\epsilon^{1/2}$ .
 By Lemma \ref{meeting}, $H$ and $e$ intersect at their midpoint $M$. Since 

\centerline{$\vert d(P,S)-d(P,M)\vert \leq d(S,M)=H$}

\noindent we also have
$\vert d(P,S)-L/2\vert \in o(1)$, which proves the third claim.\\
  
 The perpendicular arcs $H, h$ and $h'$ decompose  
 the pair of pants $\Pi(k,\epsilon)$ into two isometric   right-angled hexagons 
 (see \cite{Bu}, Proposition 3.1.5) and let  ${\bf A}$ be one of them.\\
 
 Set $d=l(D)/2$. In an anticlockwise direction, the hexagon ${\bf A}$ has sides $h', kL/2, H, kL/2, h$ and $d$. By the law of sines
 
 $${\sinh kL/2\over \sinh h'}=
 {\sinh kL/2\over \sinh h}$$
 
 \noindent and therefore $h=h'$. \\
 
 In order to prove the final statement, we first estimate $\cosh d$. By the law of cosines, we have
 
 $$\cosh H={\cosh^2 kL/2 + \cosh d
 \over \sinh^2 kL/2}$$
 
 $$=1+{1+\cosh d\over \sinh^2 kL/2}.$$
 
 It follows that
 
 $$1+\cosh d=\sinh^2 kL/2 (\cosh H -1),$$
 
 \noindent and 
 $\cosh d\sim * e^{kL} H^2\sim *\epsilon^{1-k}$.
 By the law of sines, we have
 
 $$\sinh h={\sinh H \sinh kL/2\over \sinh d}$$
 
 \noindent and therefore 
 $h\sim \sinh h\sim * \epsilon^{(k-1)/2}$.
 \end{proof}
 
  \bigskip\noindent
 {\it 4.4 The angles $\omega_i$ and $\omega_i'$}
 
 \noindent
 By Lemma \ref{geodesicD}, the geodesic $D$
 meets the arcs $e_i$ and $e_i'$ exactly once each,
 so we can define
 
 \centerline{$\omega_i:=\angle De_i$ and 
 $\omega'_i:=\angle De'_i$.}
 
 \begin{lemma}\label{increasing} We have
 \begin{enumerate}
 \item $\omega_i=\omega'_i$, for any $1\leq i\leq k-1$.
 
\item $\omega_1<\omega_2<\dots<\omega_{k-1}$.
 \end{enumerate}
\end{lemma}

\begin{proof} By definition of $\Sigma(k,\epsilon)$ and $\Pi(k,\epsilon)$, there is a isometric rotation of angle $\pi$ around the midpoint $M$ of the arc $PP'$. It follows that $\omega_i=\omega'_i$ for any $i$.\\

The pair of arcs $e_1$ and $e_{k-1}$ decompose $\Pi(k,\epsilon)$ into two connected components. Let ${\bf Q}_0$ be the contractible component,
which is a quadrilateral whose sides are
$e_1$, $e_{k-1}$, an arc of $C$ and an arc of $D$.
It is larger than the quadrilateral ${\bf Q}$ of Figure \ref{manycolours}, because the top edge is $e_1$ instead of $h$.\\

 Let 
$\pi: {\bf Q}_0\to\HP$ be a isometric embedding.
Set $\Delta'=\pi(D\cap {\bf Q})$ and let $\Delta$ be the line of $\HP$ containing the arc  $\pi(C\cap {\bf Q}_0)$. The arcs $\pi(e_1),\dots, \pi(e_{k-1})$ belongs to the $\epsilon$-pencil $\mathcal{F}_\epsilon(\Delta)$ of the line $\Delta$. The second claim therefore follows from Lemma \ref{pencil} (1).
\end{proof}

\bigskip\noindent
 {\it 4.5 Angle estimates}
 
 \begin{lemma}\label{angle} 
 For any $i=1,2,\dots,k-1$, we have
 
 $$\lim_{\epsilon\to 0} \omega_i=
 \lim_{\epsilon\to 0} \omega'_i=0 .$$
 \end{lemma}

\begin{proof} By Lemma \ref{increasing}, it is enough to show that $\lim_{\epsilon\to 0} \omega_{k-1}=0$.\\

The pair of arcs $h$ and $e_{k-1}$ cut $\Pi(k,\epsilon)$ into two connected components, where ${\bf Q}$ is the contractible component,
 as shown in Figure \ref{manycolours}. Then ${\bf Q}$ is a convex quadrilateral whose vertices are the endpoints of $h$, namely   $N:=h\cap C$ and $\Omega:=h\cap D$ and the endpoints of $e_{k-1}$, namely $\Omega_{k-1}:=e_{k-1}\cap D$ and $P_{k-1}$. \\

Let $v$ be the diagonal of ${\bf Q}$ joining $P_{k-1}$ and $\Omega$. Set

$\epsilon^-=\angle Cv$, 
$\epsilon^+=\angle ve_{k-1}$,
$\gamma^-=\angle hv$ and 
$\gamma^+=\angle vD$. By definition we have

\centerline{$\epsilon^-+\epsilon^+=\epsilon$, and
$\gamma^-+\gamma^+=\pi/2$.}

\begin{figure}[ht!]
\centering
\includegraphics[width=90mm]{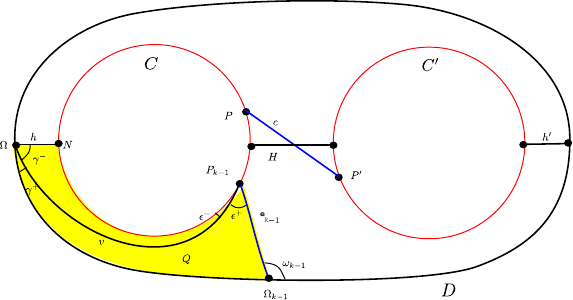}
\caption{This figure illustrates the notation from the proofs of Lemmas \ref{angle} and \ref{increasing}.}
\label{manycolours}
\end{figure}

First, we look at the trigonometry of the triangle
$P_{k-1}N\Omega$. Set $a=d(N,P_{k-1})$. We have
$a=(k/2-1)L + d(S,P)$. By Lemma \ref{chH},
we have $d(S,P)=L/2+o(1)$, therefore 
$a=uL+o(1)$, where $u=(k-1)/2$. It follows that

$$\cosh a\sim *\epsilon^{-u}.$$

Since $\cosh v=\cosh a\cosh h$, it follows from
Lemma \ref{length} that

$$\cosh v\sim\cosh a\sim *\epsilon^{-u}.$$

\noindent By the sine law, we have 
$\sin \epsilon^-=\sinh h/\sinh v$. It follows from
Lemma \ref{length} and the previous estimate that
$\epsilon^-\sim *\epsilon^{2u}$. Since $k\geq 3$,
we have $\epsilon^-<<\epsilon$ and therefore

$$\epsilon^+\sim\epsilon.$$

By combining the cosine and sine laws, we have
$\cos \gamma^-=\sin\epsilon^- \cosh a
=\sinh h \cosh a/\sinh v$, and therefore

$$\cos \gamma^-\sim *\epsilon^u.$$

Next, we will look at the trigonometry of the triangle $P_{k-1}\Omega\Omega_{k-1}$. Since
$\sin \gamma^+=\cos\gamma^-$, we have

$$\gamma^+\sim\sin\gamma^+\sim *\epsilon^u.$$

Using the cosine law, we have 
$\cos{\overline\omega_{k-1}}
=\sin\epsilon^+ \sin\gamma^+ \cosh v
-\cos \epsilon^+ \cos\gamma^+)$. Adding one on each side and using that 
$\cos{\overline\omega_{k-1}}=-\cos\omega_{k-1}$, we obtain

$$1-\cos\omega_{k-1}
=\sin\epsilon^+ \sin\gamma^+ \cosh v +
(1-\cos \epsilon^+ \cos\gamma^+).$$

We will now estimate the right term of the previous identity. 
Using a Taylor expansion, it is clear
that 

$$1-\cos \epsilon^+ \cos\gamma^+
\in O(\epsilon^2)+O(\epsilon^{2u})=O(\epsilon^2).$$

\noindent On the  other hand, we have

$$\sin\epsilon^+ \sin\gamma^+ \cosh v \sim *\epsilon.$$

\noindent Hence we have $1-\cos\omega_{k-1}\sim*\epsilon$,
i.e.

$$\omega_{k-1}\sim *\epsilon^{1/2},$$

\noindent and therefore we have proved that
$\lim_{\epsilon\to 0} \omega_{k-1}=0$.

\end{proof} 
 
\section{Sanki's paths and curve duality}\label{Gloop}

\noindent
For the whole section, we assume given an oriented topological closed manifold $\mathcal{S}$ 
of genus $g$ endowed with an isomorphism
$\rho:\Pi_g\to\pi_1({\mathcal S})$, defined modulo 
the inner conjugations. It will be called the {\it marking} of the topological surface $\mathcal{S}$.

We will consider a set $\mathrm{Tess}(\mathcal{S})$ of hexagonal tesselations of $\mathcal{S}$, which are defined by a set $\mathcal{C}_R$ of red curves and a set $\mathcal{C}_B$ of blue curves. 
Following an idea of {\cite S}\cite{FB20},
we define, for each $\tau\in \mathrm{Tess}(\mathcal{S})$, a path, called a Sanki's path, $\sigma:]0,2\pi[\rightarrow \mathcal{T}_{g}$. 
Intuitively, Sanki's path are infinitesimal analogs
of Penner's construction \cite{P88} of quasi-Anosov homeomorphisms.

When $\tau$ satisfies some additional properties, we define, for each blue curve $B$ a dual object $B^*$, which is a linear combination of three curves with coefficients  $\pm1/2$. Of course, $B^*$ is not
a multicurve, but its length function $L(B^*)$ is well defined.
The first result of the paper is Theorem \ref{sanki}, showing a kind of duality between $B$ and $B^*$. It is expressed in terms   of the Poisson
bracket $\{L(A),L(B^*)\}$ relative to the Weil-Petersson symplectic form \cite{wolpert83}.

\bigskip\noindent
{\it 5.1 The set of hexagonal tesselations
$\mathrm{Tess}(\mathcal{S})$ } 

\noindent
Let $\mathcal{H}$ be an oriented topological hexagon whose six sides are
alternatively coloured in red and blue. Strictly speaking $\mathcal{H}$ is a closed disc whose boundary is divided into six components, but the terminology hexagon is more suggestive.\\

Let $\mathrm{Tess}(\mathcal{S})$ be the set of all tesselations $\tau$ of $\mathcal{S}$ satisfying the following two axioms:
 
 \smallskip
(AX1)  The tiles are homeomorphic to $\mathcal{H}$ and they are glued pairwise along edges of the same colour.

\smallskip
(AX2)  Each vertex of the tesselation has valence four.

\smallskip
The last axiom implies that each vertex is the endpoint of four edges, which are alternately red and blue.
The graph consisting of red edges is a disjoint union of cycles. Those cycles are called the
{\it red curves of the tesselation} and the set of red curves is denoted $\mathcal{C}_{red}$. Similarly, we define the 
{\it blue curves of the tesselation} and the set $\mathcal{C}_{blue}$ of blue curves.  The set $\Curv(\tau):=
\mathcal{C}_{red}\cup \mathcal{C}_{blue}$ is called the set of {\it curves of the tesselation}.

\bigskip\noindent
{\it 5.2 Sanki paths}

\noindent We will now define the Sanki's path of a tessalation $\tau\in \mathrm{Tess}(\mathcal{S})$.
Let $\epsilon\in]0,\pi[$. Define a metric on  the $1$-skeleton $\tau_1$ of $\tau$  by requiring that all edges have length $L$. 
Recall that $L=\arcosh(1+1/\sin\epsilon)$ is 
the side lengths of  the hexagon $H(\epsilon)$ defined in Subsection 2.3.\\

For each closed face $f$ of the tesselation, let $\phi_f:H(\epsilon)\to f$ be a homeomorphism  such that its restriction to the boundary $\delta f:\partial H(\epsilon)\to\partial f$ preserves the metric and the colour of the edges.

A tesselation of $\mathcal{S}$ is obtained where each tile is endowed with a hyperbolic structure.
Along each  edge  of $\tau_1$, two geodesic arcs have been glued isometrically. Around each vertex of
$\tau_1$, the four angles are alternatively $\epsilon$ and $\overline{\epsilon}$, hence their sum
 is $2\pi$. By Theorem 1.3.5 of \cite{Bu}, there is a hyperbolic metric on $\mathcal{S}$ extending the  metric of the tiles. Together with the marking $\rho$ of ${\mathcal S}$, we obtain a well defined
 marked hyperbolic surface   $S_\tau(\epsilon)$.

 The idea of deforming right-angled regular polygons by polygons with angles of value alternatively $\epsilon$ and $\pi-\epsilon$ first appeared in \cite{S} and it was used in \cite{FB20}. Therefore the corresponding  path $\sigma_{\tau}:]0,\pi[\rightarrow \mathcal{T}_{g}$, $\epsilon\mapsto S_{\tau}(\epsilon)$ will be called the {\it Sanki's path} of the tessalation $\tau$.
 Since the function $\cosh L=1+1/\sin\epsilon$ is analytic, the path $\sigma_\tau$ is analytic. \\
 
 It should be noted that, around each vertex the colours, blue or red, and the angles, $\epsilon$ or $\overline\epsilon$, of the edges alternate. Therefore the blue curves and the red curves are geodesics with respect to the hyperbolic metric on $S_\tau(\epsilon)$.\\

\bigskip\noindent
{\it 5.3 $k$-regular tesselations}

\noindent
For a closed oriented surface $\mathcal{S}$, $(AX1)$ and $(AX2)$ is the minimal set of axioms required to define the Sanki's path. We will now define more axioms. The axiom (AX3) will ensure that the curves have the same length, while the axioms (AX4) and (AX5) are connected with the duality construction.

Let $k\geq 2$ be an integer and let $\mathcal{S}$ be a closed surface. A tesselation $\tau\in \Tess(\mathcal{S})$ is called a {\it $k$-regular tesselation} iff it satisfies the following  axiom 

\smallskip
(AX3) Each curve of $\tau$, blue or red, consists of exactly $k$ edges.

\smallskip\noindent Denote by Tess$(\mathcal{S},k)$ the set of all $k$-regular tesselations.
For any $k$-regular tesselation $\tau$, we will consider  two additional axioms. The first
axiom is  

\smallskip
(AX4) A blue edge and a red curve meet at most once.

\smallskip\noindent
Assume now that $\tau\in\mathrm{Tess}(\mathcal{S},k)$ satisfies (AX4).
Let $R$ be a red curve, let $b, b'$ be two blue edges adjacent to $R$ and let $N$ be a small regular neigborhood of $R$. Since $\mathcal{S}$ is oriented, $N\setminus R$ consists of two open   annuli, $N^{\pm}$. By axiom (AX4), $b$ has only one endpoint in $R$, therefore $b$ intersect either $N^+$ or $N^-$. Similarly, $b'$ intersect either $N^+$ or $N^-$.
We say that $b, b'$ are {\it adjacent on the same side} of 
$C$ if they both intersect $N^+$ or if they both intersect $N^-$. Our last axiom is

\smallskip\noindent
(AX5) Two distinct blue edges are adjacent on the same side of at most one red curve.

\smallskip
\noindent Denote by
$\mathrm{Tess}_{45}(\mathcal{S},k)$  the set of $k$-regular tesselations satisfying the axioms (AX4) and (AX5).

\bigskip\noindent
{\it 4.4 The isometric embedding
$\pi_b:\Pi(k,\epsilon)\to S_\tau(\epsilon)$ } 

\noindent
 From now on, assume that $k\geq 3$. Let $\tau\in \mathrm{Tess}_{45}(\mathcal{S},k)$. In order to define the duality, we first associate to each blue edge a pair of pants $\Pi(k,\epsilon)\subset S_\tau(\epsilon)$.

 Let $b$ be a blue edge with endpoints $Q$ and $Q'$, and let $R$ and $R'$ be the red curves passing through $Q$ and $Q'$.
By axiom (AX4), the two curves $R$ and $R'$ are distinct, so the graph 
$\Gamma_0:=R\cup R'\cup b$
is a union of two circles connected by an edge.
Since $\mathcal{S}$ is oriented, a small 
normal open neighbourhood $N$ of $\Gamma_0$ is a thickened eight.
Then $R\cup R'$ cuts $N$ into three components,
two of them are homeomorphic to annuli and the third one, call it $\Omega$, contains $b$.\\

In a planar representation of $\Gamma_0$, $\Omega$ is the exterior of $R\cup R'$. Since $R$ and $R'$ are
boundary components of $\Omega$, these curves inherit an orientation.
By axiom (AX3), $R$ contains $k$ vertices of the tesselation.
Starting from $Q$ in the anticlockwise direction, the other $k-1$ points of $R$ are 
denoted by $Q_1,\dots, Q_{k-1}$. For each $1\leq i\leq k-1$ let $b_i$ be the blue edge starting at 
$Q_i$ on the same side as $b$.
The points 
$Q'_1,\dots, Q'_{k-1}$ of 
$R'$  and the edges
$b'_i$ are defined similarly. Adding the edges
$b_i, b'_i$ to the graph $\Gamma_0$, we obtain a graph $\overline\Gamma$. \\

Let $\epsilon\in]0,\pi[$ and recall that $S_\tau(\epsilon)$ is the tesselated surface $S_{g}$ representing a point in $\mathcal{T}_{g}$.\\

Let $\Gamma$ be the graph defined in Section 4.1. There is a local isometry $\pi:\Gamma\to \overline\Gamma$ such that
$\pi(P)=Q$, $\pi(P_i)=Q_i$, $\pi(e_i)=b_i$, $\pi(C)=R$, $\pi(P')=Q'$, $\pi(P'_i)=Q'_i$, $\pi(e_i)=b_i$ and $\pi(C')=R'$. Clearly, it can be extended uniquely to a local isometry $\pi: \Pi(k,\epsilon)\to S_\tau(\epsilon)$. Let $\pi_{b}$ be its restriction to $\Pi(k,\epsilon)$.\\

\begin{lemma}\label{isometry} The  map 
$\pi_{b}: \Pi(k,\epsilon)\to S_\tau(\epsilon)$ is an isometric embedding.
\end{lemma}

\begin{proof} Since $b$ joins $R$ and $R'$, it follows from Axiom (Ax4) that $R$ and $R'$ are distinct. By Axioms (Ax4) and (Ax5), the point $Q_i$, respectively $Q'_j$  is the unique endpoint of $b_i\cap\Omega$, respectively of $b'_j\cap\Omega$. Hence the blue edges $b$, $b_i$ and $b'_j$ are all distinct. For any two edges $e\neq e'$ of 
$\Gamma\cap \Pi(k,\epsilon)$, we therefore have $\pi(e)\neq \pi(e')$.\\

Let $F,F'$ be two faces of $\Pi(k,\epsilon)$ such that $\pi(F)= \pi(F')$. Since each face $F$ or $F'$ has at least two blue edges in $\Gamma$, it follows that $F$ and $F'$ contain a common blue edge $e\subset \Gamma$. It follows easily that $F=F'$.

Consequently, the restriction of $\pi$ to $\Pi(k,\epsilon)\setminus\mathcal{D}$ is injective.
By Lemma \ref{geodesicD}, $\Pi(k,\epsilon)$ lies in $\Sigma(k,\epsilon)\setminus\mathcal{D}$. Therefore
$\pi$ induces an isometric embedding 
$\pi_b:\Pi(k,\epsilon)\to S_\tau(\epsilon)$.
\end{proof}

 \bigskip\noindent
{\it 5.5 The dual functions $L(B^*)$} 

\noindent
Let $\tau\in \mathrm{Tess}_{45}(\mathcal{S},k)$ for some integer $k\geq 3$.

  We are now going to define
the dual function $L(B^*)$, for any   blue curve of $B$ of the tesselation. Choose anf fix one edge $b$ of $B$.  Let $R$, $R'$ be the two red curves containing the endpoints of $b$ and set $D_b=\pi_b(D)$. Set

$$L(B^*)={1\over 2} (L(R)+L(R')-L(D_b)).$$

\noindent

 Informally speaking, $L(B^*)$ is the length function  associated with the ``dual curve" 
 $B^*= 1/2(R+R'-D_b)$. Strictly speaking, the function $L(B^*)$ depends of the choice of the edge $b$.

For any $F, G\in C^\infty(\mathcal{T}_{g})$, let
$\{F,G\}$ be their Poisson bracket induced by the
the Weil-Petersson symplectic form on 
$\mathcal{T}_{g}$, see e.g. \cite{wolpert83}. 
The duality between $B$ and  $B^*$ is demonstrated in the next lemma.

\begin{lemma}\label{dual} Let 
$\tau\in \mathrm{Tess}_{45}(\mathcal{S},k)$ and let
$\sigma_\tau:]0,\pi[\rightarrow \mathcal{T}_{g}$ be the associated Sanki path.

For any $A,\,B\in\mathcal{C}_{blue}$, we have

$$\lim_{\epsilon\to 0} \{L(A), L(B^*)\}
(\sigma_\tau(\epsilon))=
\delta_{A,B},$$

\noindent where $\delta_{A,B}$ is the Kronecker delta.
\end{lemma}

\begin{proof} Let $B\in\mathcal{C}_{blue}$. By definition there is an edge $b$ of $B$ such that 
$2 L(B^*)=L(R)+L(R')-L(D_b)$ where $R$ and $R'$ are the two red curves containing the endpoints of $b$. \\

Set $\overline{\Pi}=\pi_b({\Pi}(k,\epsilon))$ and $\overline{D}=\pi_b(D_b)$. For each $i\in\{1,2,\dots,k-1\}$, set

\centerline{$\beta_i=c_i\cap \overline{\Pi}(k,\epsilon)$ and $\beta_i'=c_i'\cap \overline{\Pi}(k,\epsilon)$.}

\noindent By Lemma \ref{isometry}, $\beta_i$ is an arc, with one endpoint $P_i$ in $R$ and the other endpoint $\Omega_i$ on $\overline{D}$. Similarly, $\beta_i'$ is an arc, with one endpoint $Q_i'$ in $R'$ and the other endpoint, say $\Omega_i'$, belongs to $\overline{D}$. We have 

\begin{enumerate}
\item $\beta_i$ does not intersect $R'$,

\item $\beta_i\cap R=Q_i$ and
$\angle R \beta_i=\epsilon$

\item $\beta_i\cap \overline{D}=\Omega_i$ 
$\angle \overline{D} \beta_i=\omega_i$,
\end{enumerate}
\noindent where the angles $\omega_i$ are defined in Section 4.4. Similarly, we have
\begin{enumerate}
\item $\beta_i'$ does not intersect $R$,

\item $\beta_i'\cap R'=P_i'$ and
$\angle R \beta_i=\epsilon$

\item $\beta_i'\cap \overline{D}=\Omega_i$ 
$\angle \overline{D} \beta_i'=\omega_i'$.
\end{enumerate}

Let $A\in\mathcal{C}_{blue}$ be another blue curve.
Set 

\centerline{$I=\{i \mid 1\leq i\leq k-1 \,
\mathrm{and} \,  \beta_i\subset A\}$, and}

\centerline{$I'=\{i\mid 
1\leq i\leq k-1 \,\,
\mathrm{and} \beta_i'\subset A\}$.}

When $A\neq B$, the curve $A$ meets 
$R\cup R'\cup D_b$ exactly at the points
$P_i, \Omega_i$ for $i\in I$ and 
$P_i', \Omega_i'$ for $i\in I$.
Therefore by Wolpert's formula \cite{W}, we have

\hskip1cm $\{L(A), L(R)+L(R')-L(D_b)\}(\sigma_\tau(\epsilon))$

\hskip1cm$=[\sum_{i\in I} \cos\epsilon-\cos\omega_i]
+ [\sum_{i\in I'} \cos\epsilon-\cos\omega_i']$.

\noindent By Lemma \ref{angle}, we have

$$\lim_{\epsilon\to0}\omega_i=0,$$ 

\noindent and therefore 

$$\lim_{\epsilon\to 0} \{L(A), L(B^*)\}
(\sigma_\tau(\epsilon))=0$$

When $A=B$, the computation is similar except
that, in addition to the arcs $\beta_i$ for $i\in I$ and $\beta'_i$ for $i\in I'$,  the geodesic $A$ contains $b$. Therefore, one obtains

\hskip1cm $\{L(A), L(R)+L(R')-L(D_b)\}
(\sigma_\tau(\epsilon))$ 

\hskip1cm$=2\cos \epsilon+
[\sum_{i\in I} \cos\epsilon-\cos\omega_i]
+ [\sum_{i\in I} \cos\epsilon-\cos\omega_i']$,

\noindent and therefore 

$$\lim_{\epsilon\to 0} \{L(A), L(A^*)\}
(\sigma_\tau(\epsilon))=1.$$
\end{proof}

\bigskip\noindent
{\it 5.6 The duality theorem}

\noindent
Suppose $k\geq 3$ and choose  $\tau\in \mathrm{Tess}(\mathcal{S},k)$. Recall that 
 $\sigma_\tau:]0,\pi[\to \mathcal{T}_{g}$ is the Sanki path.

\begin{thm}\label{sanki} Assume that $\tau$ satisfies the axioms 
(AX4) and AX(5). Then for any 
$\epsilon\in]0,\pi[$ outside some finite set $F$, the set 

$$\{\d L(C)\mid C\in \Curv(\tau)\}$$

\noindent is linearly independent at the point 
$\sigma_\tau(\epsilon)$.
\end{thm}

\begin{proof} For $\epsilon\in ]0,\pi[$, let $\delta(\epsilon)$ be the determinant of the square matrix

$$(\{L(A), L(B^*)\}(\sigma_\tau(\epsilon)))_{A,B\in\mathcal{C}_{blue}},$$

\noindent and set 
$F=\{\epsilon\in]0,\pi[\,\mid \delta(\epsilon)=0\}$.

By Lemma \ref{dual}, we have 
$\lim_{\epsilon\to0}\,\delta(\epsilon)=1$. Moreover,
changing the orientation of $\mathcal{S}$ amounts to replacing $\epsilon$ by $\overline\epsilon$, so we also have
$\lim_{\epsilon\to\pi}\,\delta(\epsilon)=\pm 1$.
Since $\delta$ is an analytic function on $]0,\pi[$, it follows that $F$ is finite.

It remains to show that, for $\epsilon\notin F$, the differentials at $\sigma_\tau(\epsilon)$ of the set of length functions 
$\{L(C)\mid C\in \Curv(\tau)\}$ are linearly independent.

Let $\epsilon \notin F$.  Let 
$(a_A)_{A\in\Curv(\tau)}$ be an element of $\R^{|\Curv(\tau)|}$
such that 

$$\sum_{A\in\Curv(\tau)}\, a_A dL(A)\vert_{\sigma_\tau(\epsilon)}=0.$$

Let $B\in \mathcal{C}_{blue}$. Recall that $B^*$ is a linear combination of two red curves $R, R'$ and
a certain geodesic $D_b$. Neither the geodesics 
$D_b$ nor the red curves 
meet  any red curve transversally. Hence we have 
$\{L(C),L(B^*)\}=0$ for any red curve $C$. It follows that

$$\sum_{A\in \mathcal{C}_{blue}}
\{ a_A L(A), L(B^*)\}$$

\noindent is zero at $\sigma_\tau(\epsilon)$. Since
$\delta(\epsilon)\neq 0$, we have
$a_A=0$ for any $A\in\mathcal{C}_{blue}$.

Therefore, it follows that
$$\sum_{A\in \mathcal{C}_{red}}\, a_A dL(A)\vert_{\sigma_\tau(\epsilon)}=0.$$  

\noindent  Since it is a subset  of some Fenchel-Nielsen coordinates, the set of differentials $\{\d L(A)\mid A\in\mathcal{C}_{red}\}$ is linearly independent. Therefore we also have 
$a_A=0$ for any $A\in\mathcal{C}_{red}$.

Hence the differentials 
$(\d L(A))_
{A\in\Curv(\tau)}$
are linearly independent at $\sigma_\tau(\epsilon)$. 
\end{proof}


\section{Local structure of ${\mathcal P}_g$ along a Sanki's path}
\label{sub:Schmutz}

\noindent
Our previous work \cite{IM1} provides a systematic construction of $k$-regular tesselations $\tau$. 
To apply Theorem \ref{sanki}, we first show
that the axioms $(AX4)$ and $(AX5)$ are automatically satisfied when 
the curves of $S_\tau(\pi/2)$ are  the
systoles.  

Then we deduce the local structure of the Thurston's spine along a Sanki's path.

To prove the main result, we will apply 
Theorem \ref{sanki} to the tesselations
$\tau_g$ of the surface $S_g$ obtained in \cite{IM1}.
Obviously they satisfy the axioms $(AX1-3)$ and it remains to prove that the tesselations $\tau_g$ also satisfy $(AX4)$ and $(AX5)$.

\bigskip\noindent
{\it 6.1 Verification of the axioms $(AX4)$ and $(AX5)$ }

\noindent
As usual, suppose $\mathcal{S}$ is a surface of genus $g\geq 2$, $k\geq 3$ and $\tau\in\Tess(\mathcal{S},k)$.

\begin{lemma}\label{axioms45}
Assume that the set of systoles of 
$S_\tau(\pi/2)$ is the set of curves of $\tau$.
 Then the tesselation $\tau$ satisfies the axioms (AX4) and (AX5).
\end{lemma}

\begin{proof} By a well-known lemma of riemannian geometry, two distinct systoles intersect in at most one point, therefore (AX4) is satisfied.\\

The proof of Axiom (AX5) is more delicate. Let $C,C'\in\mathcal{C}$ and let $c,d$ be two distinct blue edges connecting $C$ and $C'$ on the same side. Let $(P,P')\in C\times C'$ be the endpoints of $c$, and let $(Q,Q')\in C\times C'$ be the endpoints of $d$, as shown in Figure \ref{deadbug}.\\

Since $c$ and $d$ are adjacent to $C$ on the same side, there is a planar representation of $C$ and $C'$ where 
$c$ and $d$ are on the exterior of $C$. This planar representation provides an orientation of $C$ and $C'$, called the {\it direct} orientation. \\

Let $F_1$ and $F_2$ be the two hexagons containing 
$d$. Set $f_1=F_1\cap C$, $f_2=F_1\cap C$, 
$f_1'=F_1\cap C$, $f_2'=F_1\cap C$. 
By definition, $f'_1, d$ and $f_1$ are consecutive edges of $F_1$. We can assume $f'_1, d$ and $f_1$ are
ordered relative to the direct orientation. Consequently $f_2$ follows $f_1$ relative to the orientation of $C$. Since $\mathcal{S}$ is oriented,  $f'_1$ follows $f'_2$ in the direct orientation of $C'$. Therefore the relative position of $F_1$ and $F_2$ along $C$ and $C'$ is as in Figure \ref{deadbug}.\\

\begin{figure}[ht!]
\centering
\includegraphics[width=120mm]{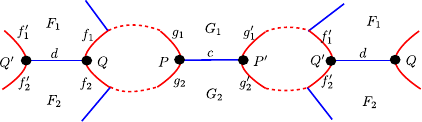}
\caption{Respective positions of $F_1$ and $F_2$.
They appear on the left side and the right sides of the figure: it should be understood that they lie
on a cylinder. }
\label{deadbug}
\end{figure}

For $i=1$ or $2$, let  $\gamma_i$ be the arc of $C$ 
from  $Q$ to $P$ and containing $f_i$. Similarly
let $\gamma_i'$ be the arc of $C'$ 
from  $P'$ to $Q$ and containing $f_i'$. Let us orient $c$ from $P$ to $P'$ and $d$ from $Q'$ to $Q$.\\

The arc $\gamma_1$, $\gamma_2$, $\gamma_1'$ and  
$\gamma_2'$ cover $C\cup C'$, therefore we have 

\centerline{$L(\gamma_1)+L(\gamma_1')
+L(\gamma_2)+L(\gamma_2')=2kL$.}

\noindent It is possible to assume without loss of generality that

\centerline{$L(\gamma_1)+L(\gamma_1')\leq kL$.}

The path $\gamma_1$ consists of 
$L(\gamma_1)/L$ edges. Let $g_1$ be the last edge
of $\gamma_1$.  Similarly, let $g_1'$ be the first edge of $\gamma_1'$. By definition, 
$g_1$ contains $P$ and $g_1'$ contains $P'$.
Let $G_1$ be the hexagon with the three consecutive edges $g_1$, $c$ and $g_1'$.\\

Now $f_1\neq g_1$ and $f_1'\neq g_1'$, otherwise the edge $d$ would join $g_1$ and $g'_1$. Hence we have $F_1\neq G_1$. Therefore there is a factorization $\gamma_1=f_1*\delta_1*g_1$, where $\delta_1$ is the geodesic arc between $f_1$ and $g_1$ and where the notation $*$ stands for the concatenation of paths. Similarly, there is a factorization $\gamma_1'=f_1*\delta_1'*g_1$. \\

We show now that the loop

\centerline{$\gamma:=\gamma_1*c*\gamma_1'*d$,}

\noindent 
is not null-homotopic. Assume otherwise. Set $S_{g}=S_\tau(\pi/2)$, let $\pi:\HP\to S_{g}$ be the universal cover of $S_{g}$ and let $\tilde{\gamma}$ be a lift of $\gamma$ in $\HP$. Since $\gamma$ is composed of four geodesic arcs, and the angles between them are $\pi/2$, the lift $\tilde{\gamma}$ would bound a quadrilateral whose inner angles are all $\pi/2$ or $3\pi/2$  which is impossible. It follows that $\gamma$ is not null-homotopic.\\

The hexagon $F_1$ contains a Saccheri quadrilateral whose basis is $d$ and with feet given by $f_1$ and $f'_1$. Let $d'$ be the fourth side, oriented from $f'_1$ to $f_1$. As a path, $d'$ is homotopic to $f'_1*d*f_1$.\\

Similarly, let $c'$ be the last side, oriented from $g_1$ to $g'_1$, of the Saccheri quadrilateral in $G_1$ whose basis is $c$ and whose feet are $g_1$ and $g'_1$. Similarly, $c'$ is homotopic to $g_1*d*g'_1$.\\

Up to a reparametrization, we have

\centerline{$\gamma=\delta_1*g_1*c*g'_1*
\delta_1'*f_1'*d*f_1$.}

\noindent Hence $\gamma$ is homotopic to
$\tilde{\gamma}=\delta_1*c'*\delta_1'*d'$.
Since we have $L(c')=L(d')=L'$, we have 

\centerline{$L(\tilde{\gamma})=L(\gamma_1)+L(\gamma_1')+2L'-4L<kL$}

\noindent by Lemma \ref{saccheri}, which contradicts that $\mathcal{C}$ is the set of systoles.
\end{proof}

\bigskip\noindent
{\it 6.2 Two corollaries}

\noindent
We will now derive two corollaries concerning the structure of ${\mathcal P}_g$ at the neighborhood of a Sanki's path.

Given a finite set of curves 
${\mathcal C}=\{C_1, C_2,\dots,C_n\}$,  let 
$E({\mathcal C})$ be the set of 
$x\in {\mathcal T}_g$ such that

\centerline{$L(C_1)(x)=L(C_2)(x)=\dots=L(C_n)(x)$.}

\noindent Also let $\Sys({\mathcal C})$ be the set of points
$x\in{\mathcal P}_g$ such that ${\mathcal C}$ is the set of systoles at $x$.

Let $X\subset {\mathcal P}_g$ be a locally closed subset, and let $x\in \overline{X}$. We say that
$X$ is {\it locally a smooth manifold at $x$} if
$U\cap X$ is a smooth manifold for some open neighborhod of $X$. When it is the case,
the local codimension $\codim_x\,X$ is well defined.
Our previous definition do not require that $x$ belongs to $X$.

\begin{cor}\label{codim} Let 
$\tau\in\Tess(\mathcal{S},k)$ for some $k\geq 3$
such that $\Curv(\tau)$ is the set of
systoles of $S_\tau(\pi/2)$.

Let $\mathcal{C}\subset \Curv(\tau)$  be any filling subset. Then for any $\epsilon\neq \pi/2$ closed to $\pi/2$, we have

(i) $\sigma_\tau(\epsilon)$ belongs to
$\overline{\Sys(\mathcal{C})}$,

(ii) $\Sys(\mathcal{C})$ is a smooth manifold in the neighborhood of $\sigma_\tau(\epsilon)$, and

(iii) $\codim_{\sigma_\tau(\epsilon)}\,
\Sys(\mathcal{C})=\Card \mathcal{C}-1$.
\end{cor}

\begin{proof}  
By Lemma \ref{axioms45}, the tesselation $\tau$ belongs to
$\Tess_{45}(\mathcal{S},k)$. Hence by Theorem \ref{sanki}, the map 

\centerline{$L(\mathcal{C}):
\mathcal{T}(\mathcal{S})\to \R^\mathcal{C},
x\mapsto (L(C)(x))_{C\in\Curv(\tau)}$}

\noindent is a submersion at the point 
$\sigma_\tau(\epsilon)$
for all $\epsilon\neq\pi/2$ closed to $\pi/2$.
By the submersion theorem,  $E(\mathcal{C})$ is smooth
of codimension $\Card\,\mathcal{C}-1$ around 
the point $\sigma_\tau(\epsilon)$
and $\sigma_\tau(\epsilon)$ is adherent to
the set
$E^+({\mathcal C})$ of all $x\in E({\mathcal C})$
defined by the inequations 

\centerline{
$ L(C)(x)<L(C')(x)$,  for all  $C\in{\mathcal C}$ and $C'\in \Curv(\tau)\setminus {\mathcal C}$.}

 Thus Assertion (ii) and (iii) follows from the fact that $\Sys(\mathcal{C})$ is an open set of $E(\mathcal{C})$, see \cite{SS}\cite{T85}.
\end{proof}

\begin{cor}\label{cor2}
Under the hypothesis of Corollary \ref{codim}, the point $\sigma_\tau(\pi/2)$ is adherent to
$\Sys(\mathcal{C})$ and we have

\centerline{$\codim\, \Sys(\mathcal{C})<\Card(\mathcal{C})$.}
\end{cor}

\bigskip\noindent
{\it 6.3 The main result from \cite{IM1}}

\noindent
A {\it decoration} of the hexagon $H(\pi/2)$
is a cyclic  indexing of it six sides by 
$\Z/6\Z$. Up to direct isometries, there are exactly
two decorated hexagons, say ${\mathcal H}$ and
$\overline{\mathcal H}$. 

Let $S$ be a closed hyperbolic surface. A {\it standard} hexagonal tesselation $\tau$ of $S$ is a
tesselation of $S$, where each tile is isomorphic to 
${\mathcal H}$ or $\overline{\mathcal H}$. 
Of course, it is assumed that tiles are glued along
edges of the same index. 
 
\begin{citethm}[Theorem 25 of \cite{IM1}]
There exists an infinite set $A$ of integers $g\geq 2$, and, for any $g\in A$, a closed oriented hyperbolic surface $S_g$ of genus $g$ endowed with a standard  tesselation $\tau_g$ satisfying the following assertions

\begin{itemize}
\item[(1)] the  systoles of $S_g$ are the curves of 
$\tau_g$, and
\item[(2)] we have 
$$\Card\,\Syst(S_g)\leq {57\over \sqrt{\ln\ln\ln g}}\,\,
{g\over \sqrt{ \ln g}} .$$
\end{itemize}
\end{citethm}

\bigskip\noindent
{\it 6.4 Proof of Theorem 1}

\begin{Main1} 
There exists an infinite set $A$ of integers $g\geq 2$ such that

\centerline{$\codim\, \mathcal{P}_g <
{38\over \sqrt{\ln\ln\ln g}}\,\,
{g\over \sqrt{ \ln g}}$,}

\noindent for any $g\in A$.
\end{Main1} 

\begin{proof} Let $A$ be the set of the
of the  theorem of Subsection 6.3.
Let $g\in A$ and let  $\tau_g$ be the corresponding
tesselation. \\

By hypotheses, any curve $C$ of $\tau$ consists
of edges of the same index. By extension it will be
called the index of the curve. Let
${\mathcal C}$ be the set of all curves of index 
$3, 4,5$ or $6$. We claim that ${\mathcal C}$ fills the surface.

Let $P$ be a vertex at the intersection of 
two curves of index $1$ and $2$. Let
${\bf Q}$ be the union of the four hexagons
surrounding $P$. It turns out that 
${\bf Q}$  is a $12$-gon whose edges have indices
distinct from $1$ and $2$. It follows that
${\mathcal C}$ cuts the surface into these $12$-gons. \\

It is clear that 
$\Card\,{\mathcal C}=2/3\, \Card\,\Syst(S_g)$.
To finish the proof, it is enough to show that 
$\tau_g$ satisfies the hypothesis of Corollary \ref{cor2}.
 
We can assign the red colour to the edges 
of $\tau_g$, of index $1,2$ or $3$ and the blue colour to other edges.
Moreover, since all curves have the same 
length, the tesselation is $k$-regular for some $k$. The case $k=2$ was excluded from consideration
in \cite{IM1} so we have $k\geq 3$. In fact, the decoration implies that $k$ is  even \cite{IM1}, so 
we have $k\geq 4$. It follows that 
$\tau_g$ belongs to $\Tess(\mathcal{S},k)$ for some
$k\geq 4$.

It follows from Corollary \ref{cor2} that

\centerline{$\codim\, \mathcal{P}_g <
{38\over \sqrt{\ln\ln\ln g}}\,\,
{g\over \sqrt{ \ln g}}$.}
\end{proof}

\section{Examples}

\noindent Before \cite{FB23}, it was a challenging question to know if $\codim {\mathcal P}_g$ was less than $2g-1$. Since the bounds in \cite{FB23} are not explicit, it is still interesting to know
the smallest $g$ for which 
$\codim {\mathcal P}_g<2g-1$. We will describe our construction for $g=17$ and show that
$\codim {\mathcal P}_{17}<33$.

We will first  briefly explain the case $2k=2$ which was excluded from consideration
in order to avoid some specificity.

\bigskip\noindent
{\it 7.1 Standard $2k$-regular tesselations}

\noindent We will briefly explain 
the construction  of all standard $2k$-regular tesselations, following \cite{IM1}. 
Let ${\mathcal H}$ be a decorated right-angled regular hexagon
of the Poincar\'e half-plane $\HP$.
For each $i\in\Z/6\Z$,  
let $s_i$ be the reflection in the line $\Delta_i$
containing the side of index $i$ of ${\mathcal H}$.
The group $W$ generated by these reflections is a Coxeter group with presentation

$$\langle s_i\mid (s_is_{i+1})^2=1, \forall i\in\Z/6\Z\rangle.$$

\noindent   
By a  theorem of Poincar\'e, the
collection of hexagons 
$\{w.{\mathcal H}\}$ is the set of tiles of a tesselation of $\tau$ of $\HP$.

Let $W^+$ be the subgroup of index two consisting of  products of an even number of generators.
Let $k\geq 1$. Let $H$ be a subgroup of
$W$ satisfying

\begin{enumerate}

\item $H$ is a finite index subgroup of $W^+$,

\item $H^w\cap \langle s_i, s_{i+1}\rangle =\{1\}$,
and

\item $H^w\cap \langle s_i s_{i+1}\rangle =
\langle (s_i s_{i+1})^k\rangle$,

\end{enumerate} for any $i\in \Z/6\Z$ and $w\in W$,
where $H^w$ stands for $wHw^{-1}$.

\noindent Then $\HP/H$ is a closed oriented 
hyperbolic surface endowed with a $2k$-regular
standard tesselation. Conversely, any
such tesselated surface is  isometric to
 $\HP/H$, where $H$ satisfies the previous conditions, see \cite{IM1}, Theorem 12.
This leads to the question, only partially answered by 
Criterion 18 of \cite{IM1} - when the curves of the tesselation are the systoles of the surface?

\bigskip\noindent
{\it 7.2 Schmutz's genus two surface.}

\noindent
 The case $2k=2$ is  simple, but it has been excluded   because of its
 particularity. In fact, the three-holed sphere
 $\Sigma(2,\epsilon)$ is  equal to
 $\Pi(2,\epsilon)$. 
 
 There is only one subgroup $H$ of $W$ satisfying the previous three conditions, and we have
$W/H\simeq(\Z/2\Z)^2$. The corresponding surface
$S_2:=\HP/H$ is the genus $2$ surface tesselated by $4$ hexagons, see Figure 7. It has been proved in
\cite{SS} that the set ${\mathcal C}$ of the curves of the tesselation are the systoles. 

The curve $S_2$ has six points $P_i$,
for $i\in\Z/6\Z$, which are fixed by the 
hyperelleptic invotion. Denote by $C_{i,i+1}$
the  curve of ${\mathcal C}$ containing 
$P_i$ and $P_{i+1}$.

Tedious computations show that 
${\mathcal C}$ is the set of systoles at the point
$\sigma(\epsilon)$ for any
$\epsilon\in ]\pi/4,\pi/2$. The limit point
$\sigma(\pi/4)$ is the Bolza's surface 
with $12$-systoles \cite{Bol}. Let 
${\mathcal C}'$ be the six new systoles of 
$\sigma(\pi/4)$. Each of these systoles contains
the hyperelliptice point $P_i$ and $P_{i+2}$ for some $i$. For $\epsilon<\pi/4$,  ${\mathcal C}'$ is the set of systoles at the point $\sigma(\epsilon)$.
Since  ${\mathcal C}'$ does not fill,
$\sigma(\epsilon)$ is no more in ${\mathcal P}_g$ for $\epsilon<\pi/4$.

A similar analysis can be carried for
$\epsilon\geq \pi/2$. The limit points at $\pi/4$ and $3\pi/4$ are   Bolza's surface with  
distinct markings.

\begin{figure}[!thpb]
\centering
\includegraphics[width=0.6\textwidth]{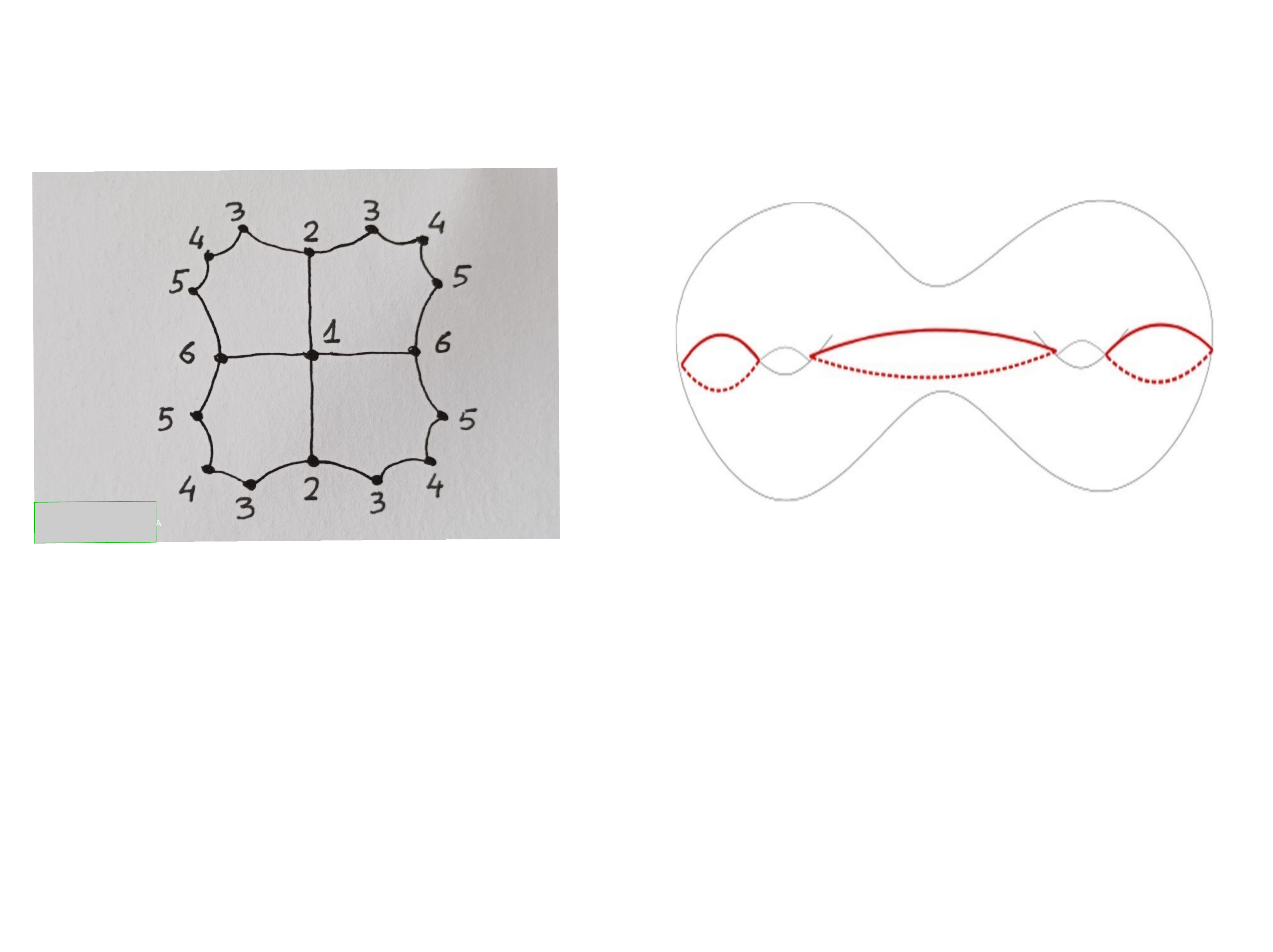}
\caption{Up to repetition, there are only six vertices on the left side of this figure, which are the  points  $P_i$ indexed by 
$1,2,\dots,6$. They are located 
on the $x$-axis of the figure on the right. The hyperelliptic involution is a $180$-degree rotation around this axis. Three systoles are located on the vertical plane and the other three are on the horizontal plane.}
\label{genus2example}
\end{figure}

 \bigskip\noindent
{\it 7.3 An exemple of genus $17$.}

\noindent When $2k=4$, the analysis is more complicated. 
We will  describe a surface of genus $17$ endowed with a $4$-regular tesselation.

Let $H\subset W$ be the normal subgroup generated by the elements $(s_i s_{i+3})^2$ and set
 $S_{17}=\HP/H$. The quotient $\Gamma:= W/H$ is isomorphicto $(\Z/2\Z)^6$, hence
$S_{17}$ is tesselated by $64$ hexagons. It follows that $S_{17}$ has genus $17$.

\begin{lemma}\label{17} The systoles of $S_{17}$ are exactly
the curves of the given tesselation.
\end{lemma}

\begin{proof} This specific example
does not fully satisfies the hypotheses of
criterion 18 of \cite{IM1}, so we will briefly explain the proof.

The group $\Gamma$ is given with $6$ generators,
and its Caley graph $\Cay\,\Gamma$ is the one-skeleton of a 
$6$-dimensional cube.
There is an embedding of $\Cay\,\Gamma$
in $S_{17}$. The  vertices are  the centers of the hexagons  and the edges are the geodeosic arcs
connecting two vertices belonging to two adjacent faces and crossing their
common edge.

A loop in  $\Cay\,\Gamma$ is a word $w$ on the letters
$(s_i)_{i\in\Z/6\Z}$ representing $1$ in $\Gamma$.
The letters $s_1,s_3$ and $s_5$ are called the red letters, and the other three are called the blue letters.
For any word $w$, let $l_R(w)$, resp.   $l_B(w)$, be the number of occurences of red letters, resp.
of blue letters. Also set 
$l(w)=l_R(w)+l_B(w)$.

As in Lemma 14 of \cite{IM1}, any closed geodesic $\gamma$ is freely homotopic to a loop
$\omega(\gamma)$  in $\Cay\,\Gamma$. 
Indeed if $\gamma$ crosses sucessively some
edges of index $i_1$, $i_2$, $\dots i_n$
then $\omega(\gamma)$ is the word 
$s_{i_1}s_{i_2}\dots s_{i_n}$. If at some point
$\gamma$ crosses a vertex at the intersection 
of two edges of indices $i$ and $i+1$, the previous definition is ambiguous. By convention, 
we will consider that $\gamma$ crosses first an edge of index $i$ and then an edge of index $i+1$. 

We claim that the systoles of 
$S_{17}$ are the curves of the tesselation,
which have length $4L$, where $L=\arcosh\,2$.
Let $\gamma$ be a closed geodesic. Note that
$l_R(\omega(\gamma)$ and  
$l_B(\omega(\gamma)$ are even. 

First assume that $l(\omega(\gamma)>4$. We have
$l_R(\omega(\gamma)\geq 4$ or 
$l_B(\omega(\gamma)\geq 4$ and
$\gamma$ is not a curve of the tesselation. Therefore
$\gamma$ has length bigger that $4L$ by
Lemma 17 of \cite{IM1}.

Next assume that $l(\omega(\gamma)\leq 4$
It is obvious that $l(\omega(\gamma)$ is bigger than $2$, so we have  $l(\omega(\gamma))=4$. 

Note that
$\omega(\gamma)$ cannot contains two identical
consecutive letters, so $\omega(\gamma)=s_i s_js_is_j$ for some $i\neq j$.  Note also that the words $s_is_{i+1}s_is_{i+i}$ are null-homotopic in
$S_{17}$.
If $\omega(\gamma)=s_i s_{i+2}s_is_{i+2}$, then
$\gamma$ is a curve of the tesselation.
If $\omega(\gamma)=s_is_{i+3}s_i s_{i+3}$,
then $\gamma$ is a concatenation of four arcs
which connects the middles of a side of index $i$ to a side of index $i+3$. If $c$ is one of these arcs, it cut an  hexagon into two right-angled pentagons.
By the formula of Theorem 3.5.10 of \cite{Rat},
we have $l(c)=\arcosh 3$, therefore
$l(\gamma)=4 \arcosh 3$ is bigger than $4L$.
Since $\gamma$ is defined
up to an orientation, we have treated all cases where $l(\omega(\gamma))=4$.
\end{proof}

The next lemma shows $\codim\,{\mathcal P}_g<2g-1$
for $g=17$.

\begin{lemma} We have
$\codim\,{\mathcal P}_{17}<33$.
\end{lemma}

\begin{proof} The surface $S_{17}$ has $48$
curves. Let ${\mathcal C}$ be the set of all curves of index 3,4,5, or 6. As in the proof of
Theorem 1, the set ${\mathcal C}$ fills the surface. Since $\Card {\mathcal C}=32$,
we have 
 $\codim\,{\mathcal P}_{17}<32$ by Corollary \ref{cor2}.
 \end{proof}

 \smallskip\noindent
{\it Remark.} 
The set ${\mathcal C}$ of the proof is not a minimal filling subset. Intuitive computations suggest that the minimal filling subsets have 
cardinality $25$, and that
$\codim\,{\mathcal P}_{17}=24$.

 \bibliography{EstimatingBib}

\begin{thebibliography}{10}

\bibitem{Bol}
O.~Bolza.
\newblock On binary sextics with linear transformations between themselves.
\newblock {\em Amer. J. Math.}, 10:47–70, 1888.

\bibitem{FB20}
M.~Fortier Bourque.
\newblock Hyperbolic surfaces with sublinearly many systoles that fill.
\newblock {\em Commentarii Mathematici Helvetici}, 95:515--534, 2020.

\bibitem{FB23}
M.~Fortier Bourque.
\newblock The dimension of {T}hurston's spine.
\newblock {\em Int. Math. Res. Not.}, pages 1--10, 2023.

\bibitem{BV}
M.~Bridson and K.~Vogtmann.
\newblock Automorphism groups of free groups, surface groups and free abelian
  groups.
\newblock In {\em Problems on mapping class groups and related topics,},
  volume~74 of {\em Proceedings of Symposia in Pure Mathematics}, pages
  301--316. American Mathematical Society, 2006.

\bibitem{Bu}
P.~Buser.
\newblock {\em Geometry and {S}pectra of {C}ompact {R}iemann {S}urfaces},
  volume 106 of {\em {P}rogress in {M}athematics}.
\newblock Birkh{\"a}user, 1992.

\bibitem{FM}
B.~Farb and D.~Margalit.
\newblock {\em A primer on mapping class groups}, volume~49 of {\em Princeton
  Mathematical Series}.
\newblock Princeton University Press, Princeton, NJ, 2012.

\bibitem{H}
J.~Harer.
\newblock The virtual cohomological dimension of the mapping class group of an
  orientable surface.
\newblock {\em Inventiones Mathematicae}, 84:157--176, 1986.

\bibitem{Calculation}
I.~Irmer.
\newblock An explicit deformation retraction of the genus 2 moduli space.
\newblock To appear.

\bibitem{Me}
I.~Irmer.
\newblock The differential topology of the {T}hurston spine.
\newblock arXiv:2211.03429, 2022.

\bibitem{IM1}
I.~Irmer and O.~Mathieu.
\newblock Small systole sets and {C}oxeter groups.
\newblock Preprint, 2023.

\bibitem{LiJi}
Lizhen Ji.
\newblock Well-rounded equivariant deformation retracts of {T}eichm\"uller
  spaces.
\newblock {\em L’Enseignement Mathématique}, 60(02):109--129, 2013.

\bibitem{L1}
S.~Lojasiewicz.
\newblock Triangulation of semi-analytic sets.
\newblock {\em Annali della Scuola Normale Superiore di Pisa - Classe di
  Scienze}, 18(4):449--474, 1964.

\bibitem{P88}
R.~C. Penner.
\newblock A construction of pseudo-{A}nosov homeomorphisms.
\newblock {\em Trans. Am. Soc.}, 310(1):179--197, 1988.

\bibitem{Rat}
J.~Ratcliffe.
\newblock {\em Foundations of hyperbolic manifolds}, volume 149 of {\em
  Graduate Texts in Mathematics}.
\newblock Springer, New York, second edition, 2006.

\bibitem{S}
B.~Sanki.
\newblock Systolic fillings of surfaces.
\newblock {\em Bulletin of the {A}ustralian {M}athematical {S}ociety},
  98:502--511, 2018.

\bibitem{SS}
P.~Schmutz Schaller.
\newblock Systoles and topological {M}orse functions for {R}iemann surfaces.
\newblock {\em Journal of Differential Geometry}, 52(3):407--452, 1999.

\bibitem{T85}
W.~Thurston.
\newblock A spine for {T}eichm\"uller space.
\newblock Preprint, 1985.

\bibitem{W}
S.~Wolpert.
\newblock An elementary formula for the {F}enchel-{N}ielsen twist.
\newblock {\em Comment. Math. Helv.}, 56(1):132--135, 1981.

\bibitem{wolpert83}
S.~Wolpert.
\newblock On the symplectic geometry of deformations of hyperbolic spaces.
\newblock {\em The Annals of Mathematics}, 117:207--234, 03 1983.

\end{thebibliography}
\bibliographystyle{plain}
\end{document}